\documentclass[12pt]{article}
\title{Transverse instability for nonlinear Schr\"odinger equation with a linear potential 
\amssubj{35B32, 35B35, 35Q55.}}
\author{YOHEI YAMAZAKI\footnote{{\it E-mail addresses:} y-youhei@math.kyoto-u.ac.jp} \\ {\footnotesize Department of Mathematics, Kyoto University,}\\ {\footnotesize Kitashirakawa-Oiwakecho, Sakyo, Kyoto 606-8502, Japan}}
\usepackage{amssymb}
\usepackage{mathrsfs}
\usepackage{amsthm}
\usepackage{amsmath}
\usepackage{titlefoot}
\usepackage{fancyhdr}
\def\pdfliteral #1 {}

\numberwithin{equation}{section}

\newtheorem{theorem}{Theorem}[section]
\newtheorem{corollary}[theorem]{Corollary}

\newtheorem{lemma}[theorem]{Lemma}
\newtheorem{proposition}[theorem]{Proposition}

\theoremstyle{definition}
\newtheorem{definition}[theorem]{Definition}
\newtheorem{remark}[theorem]{Remark}

\renewcommand{\eqref}[1]{(\ref{#1})}

\renewcommand{\bigskip}{\vspace{0.3cm}}
\renewcommand{\l}{\left}
\renewcommand{\r}{\right}

\newcommand{\R}{{\mathbb R}}
\newcommand{\C}{{\mathbb C}}

\newcommand{\Z}{{\mathbb Z}}
\newcommand{\T}{{\mathbb T}}

\newcommand{\norm}[1]{{\left \lVert #1 \right \rVert}}

\newcommand{\tbr}[1]{\langle #1 \rangle}
\newcommand{\Tbr}[1]{\left\langle #1 \right\rangle}
\newcommand{\RT}{\mathbb {R} \times \mathbb {T}}
\newcommand{\RTL}{\mathbb {R} \times \mathbb {T}_L}

\newcommand{\re}{\mbox{\rm Re }}
\newcommand{\im}{\mbox{\rm Im }}

\date{}
\begin{document}

\maketitle


\begin{abstract}
In this paper we consider the transverse instability for a nonlinear Schr\"odinger equation with a linear potential on $\RTL$, where $2\pi L$ is the period of the torus $\T_L$.
Rose and Weinstein \cite{R W} showed the existence of a stable standing wave for a nonlinear Schr\"odinger equation with a linear potential.
We regard the standing wave of nonlinear Schr\"odinger equation on $\R$ as a line standing wave of nonlinear Schr\"odinger equation on $\R \times \T_L$.
We show the stability of line standing waves for all $L>0$ by using the argument of the previous paper \cite{YY2}.
\end{abstract}


\section{Introduction}

We consider the nonlinear Schr\"odiner equation with linear potential
\begin{equation}\label{NLS}
 i\partial_t u = -\Delta u +V(x) u - |u|^{p-1}u, \quad (t,x,y) \in \R \times \RTL,
\end{equation}
where $p>1$, a potential $V: \R \to \R$ and $u=u(t,x,y)$ is an unknown complex-valued function for $t \in \R, x \in \R$ and $y \in \T_L$.
Here, $\T_L=\R/2\pi L \Z$ and $L>0$.

We assume the following conditions for $V$.
\begin{enumerate}
  \setlength{\parskip}{0.1cm} 
  \setlength{\itemsep}{0cm} 
\renewcommand{\labelenumi}{\rm (V\arabic{enumi})}
\item There exist $C>0$ and $\alpha >0$ such that $ |V(x)| \leq Ce^{-\alpha |x|} $.
\item $-\partial_x^2 +V$ has the lowest eigenvalue $-\lambda_*<0$.
\end{enumerate}

The Cauchy problem (\ref{NLS}) is locally well-posed in $H^1(\RTL)$ by using the argument in \cite{G V} and \cite{T T}.
The equation (\ref{NLS}) has the following conservation laws:
\[\begin{split}
E(u)&=\int_{\RTL}\l(\frac{1}{2}|\nabla u|^2 + \frac{1}{2}V(x)|u|^2 -  \frac{1}{p+1} |u|^{p+1}\r)dxdy,\\
Q(u)&=\frac{1}{2}\int_{\RTL} |u|^2dxdy,
\end{split}\]
where $u \in H^1(\RTL)$.

We define a standing wave $u(t)$ as a non-trivial solution of (\ref{NLS}) having the form $u(t) = e^{i\omega t}\varphi$.
Then, $e^{i\omega t}\varphi$ is a standing wave if and only if $\varphi$ is a non-trivial solution of 
\begin{equation}
-\Delta \varphi + \omega \varphi + V(x) \varphi - |\varphi|^{p-1}\varphi = 0, \quad (x,y) \in \RTL.
\end{equation}
Using the bifurcation theory, Rose and Weinstein \cite{R W} showed the existence of the stable standing wave $e^{i \omega t}\varphi_{\omega}$ for the following nonlinear Schr\"odinger equation
\begin{equation}\label{1DNLS}
 i\partial_t u = -\partial_x^2 u +V(x) u - |u|^{p-1}u, \quad (t,x) \in \R \times \R.
\end{equation}
Then, the standing wave $e^{i\omega t}\varphi_{\omega}$ satisfies the following.
\begin{proposition}\label{R-W}
Let $\psi_*$ be the eigenfunction of $-\partial_x^2 +V(x)$ corresponding to $-\lambda_*$ with $\psi_*>0$ and $\norm{\psi_*}_{L^2}=1$.
Then, there exists $\omega_*>\lambda_*$ such that for $\lambda_* < \omega < \omega_*$, $e^{i\omega t}\varphi_{\omega}$ is a stable standing wave of {\rm (\ref{1DNLS})} satisfying 
\[ \varphi_{\omega} = \norm{\psi_*}_{L^{p+1}(\R)}^{-\frac{p+1}{p-1}}(\omega-\lambda_*)^{\frac{1}{p-1}}\psi_* + r(\omega),\]
where $\norm{r(\omega)}_{H^2(\R)}=O((\omega_0-\lambda_*)^{\frac{1}{p-1}+1})$.
Moreover, $ L_{\omega}^+ = -\partial_x^2 + \omega +V - p|\varphi_{\omega}|^{p-1}$ has the exactly one negative eigenvalue $-\lambda_{\omega}$ and does not have the zero eigenvalue.
\end{proposition}
We define the line standing wave $e^{i\omega t}\tilde{\varphi}_{\omega}$  of (\ref{NLS}) as
\[ \tilde{\varphi}_{\omega}(x,y)=\varphi_{\omega}(x), \quad (x,y) \in \RTL.\]
In this paper, we consider the transverse instability of the line standing wave $e^{i\omega t}\tilde{\varphi}_{\omega}$.
The stability of standing waves is defined as follows.
\begin{definition}
We say the standing wave $e^{i\omega t}\varphi$ is orbitally stable in $H^1$ if for any $\varepsilon >0$ there exists $\delta>0$ such that for all $u_0\in H^1(\RTL)$ with $\norm{u_0-\varphi}_{H^1}<\delta$, the solution $u(t)$ of (\ref{NLS}) with the initial data $u(0)=u_0$ exists globally in time and satisfies
\[ \sup_{t\geq 0} \inf_{\theta \in \R, y \in \T_L} \norm{u(t,\cdot,\cdot-y)-e^{i\theta}\varphi(\cdot,\cdot-y)}_{H^1} <\varepsilon .\]
Otherwise, we say the standing wave $e^{i\omega t}\varphi$ is orbitally unstable in $H^1$.
\end{definition}

The transverse instability for KP-I or KP-II equation is treated in \cite{A P S,TM,M T,R T 0,R T 1,R T 2,R T 3}.
In \cite{A P S}, Alexander-Pego-Sachs studied the linear instability for line solitons of KP-I and KP-II.
In \cite{M T}, Mizumachi-Tzvetkov proved the asymptotic stability for line solitons of KP-II on $\RT$.
Modulating the local phase and the local amplitude of line solitons, Mizumachi showed the asymptotic stability for line solitons of KP-II on $\R^2$ in \cite{TM}.
Rousset-Tzvetkov proved the transverse instability for line solitons of KP-I on $\R^2$ in \cite{R T 0} and on $\RTL$ in \cite{R T 1}.
In \cite{R T 3}, Rousset-Tzvetkov showed the stability of line solitons for KP-I on $\RTL$ with small $L>0$.
Moreover, Rousset-Tzvetkov proved the existence of the critical period $4/\sqrt{3}$ for the period $L$ of the transverse direction. 
Namely, a line soliton for KP-I on $\RTL$ is stable for $0<L<4/\sqrt{3}$ and unstable for $L>4/\sqrt{3}$.

The transverse instability for a nonlinear Schr\"odinger equation is studied in \cite{D P C,R T 0,R T 1,YY1,YY2}.
In \cite{D P C}, Deconinck-Pelinovsky-Carter studied the linear stability for  line standing waves of a hyperbolic Schr\"odinger equation.
Rousset-Tzvetkov proved the transverse instability for cubic nonlinear Schr\"odinger equation without linear potential on $\R^2$ in \cite{R T 0} and on $\RTL$ in \cite{R T 1}.
To prove the instability, Rousset-Tzvetkov applied the argument of Grenier \cite{EG}.
Rousset-Tzvetkov constructed the high order approximate solution with an unstable eigenmode and showed a precise estimate of the growth of the semi-group generated by the linearized operator.
To construct the high order approximate solution, we use the regularity of the nonlinearity $|u|^2u$ in the sense of Fr\'echet differentiation.
In \cite{YY1}, the author studied the transverse instability for line standing waves of a system of nonlinear Schr\"odinger equations on $\RTL$ which was treated in \cite{C O}.
In \cite{YY1}, the existence of the critical period for a period $L$ was also proved, which was suggested by Rousset-Tzvetkov.
Constructing the estimate for high frequency parts of solutions and using the existence of local solutions, the author showed the transverse instability for line standing waves of equations with the general power nonlinearity.
In \cite{YY2}, the author considers the stability for a line standing wave of (\ref{NLS}) with $V=0$.
The application of the argument in \cite{YY1} yields the existence of the critical period for a line standing wave of (\ref{NLS}) with $V=0$.
For (\ref{NLS}) with $V=0$ and the critical period, the linearized operator around the line standing wave is degenerate.
Therefore, we can not directly apply the argument in Grillakis-Shatah-Strauss \cite{G S S 1,G S S 2}.
Since the linearized operator around the line standing wave with the critical period does not have any unstable eigenvalues, we can not show the instability by the argument based on the occurrence of unstable eigenmode in \cite{G O,G S S 1,R T 1,YY1}.
Moreover,  the third order term of the Lyapunov functional around the line standing wave with the critical period does not appear.
Thus, we can not apply the argument for the degenerate case of the stability in \cite{MO}.
The transverse instability comes from the symmetry breaking bifurcation.
In \cite{YY2}, applying the bifurcation result for symmetry breaking bifurcation and the stability result for the degenerate case in \cite{MM}, the author showed the stability for the line standing wave with critical period for some exponents $p\geq 2$ of the nonlinearity.

The followings are our main theorems in this paper.
In the first theorem, we show the transverse instability of the line standing wave $e^{i\omega t}\tilde{\varphi}_{\omega}$ and obtain the critical period between the stability and the instability. 
\begin{theorem}\label{main result 1}
There exists $\omega_{*,0}>\lambda_*$ such that for $\lambda_*<\omega< \omega_{*,0}$ the followings two assertions hold:
\begin{enumerate}
  \setlength{\parskip}{0.1cm} 
  \setlength{\itemsep}{0cm} 
\renewcommand{\labelenumi}{\rm (\roman{enumi})}
\item If $0< L < (\lambda_\omega)^{-\frac{1}{2}}$, then the standing wave $e^{i\omega t}\tilde{\varphi}_{\omega}$ is stable.
\item If $ (\lambda_\omega)^{-\frac{1}{2}}<L$, then the standing wave $e^{i\omega t}\tilde{\varphi}_{\omega}$ is unstable.
\end{enumerate}
\end{theorem}
In the Second theorem, we show the stability for the line standing wave $e^{i\omega t}\tilde{\varphi}_{\omega}$ with the critical period $L=(\lambda_\omega)^{-1/2}$.
\begin{theorem}\label{main result 2}
Let $p \geq 2$ and 
\[p_* = \frac{9+ \sqrt{57}}{4}.\]
Then there exists $\lambda_*< \omega_p $ satisfying the following two properties:
\begin{enumerate}
  \setlength{\parskip}{0.1cm} 
  \setlength{\itemsep}{0cm} 
\renewcommand{\labelenumi}{\rm (\roman{enumi})}
\item If $p< p_*$ and $\lambda_*< \omega < \omega_p$, then the standing wave $e^{i\omega t}\tilde{\varphi}_{\omega}$ of {\rm (\ref{NLS})} with $L=(\lambda_\omega)^{-1/2}$ is stable.
\item If $p_*\leq p$ and $\lambda_*< \omega < \omega_p$, then the standing wave $e^{i\omega t}\tilde{\varphi}_{\omega}$ of {\rm (\ref{NLS})} with $L=(\lambda_\omega)^{-1/2}$ is unstable.
\end{enumerate}
\end{theorem}

The proof of Theorem \ref{main result 1} follows form the spectrum analysis of the linearized operator  and  the estimate of high frequency parts of solution by the argument in \cite{YY1}.
To show the growth of the semi-group generated by the linearized operator, we use the assumption of the decay for the linear potential $V$.
For the proof of Theorem \ref{main result 2}, we apply the bifurcation analysis for the symmetry breaking bifurcation and the argument for the stability in \cite{YY2}.
In \cite{YY2}, to prove the stability for the line standing wave with the critical period, we show the increase of $L^2$-norm of the symmetry breaking standing wave with respect to the bifurcation parameter or the decrease of it.
To show the increase of $L^2$-norm, we need to calculate an integral of a solution of an ordinary differential equation which comes from the linearized equation of one dimensional Schr\"odinger equation around a standing wave.
Since it is difficult to obtain the explicit solution of the ordinary differential equation in the argument in \cite{YY1}, we can not calculate the exact value of the integral and we estimate the value of the integral.
Therefore, it is not known whether the line standing wave is stable or unstable for some nonlinear Schr\"odinger equations with the power nonlinearity $|u|^{p-1}u$   which has some exponent $p \in (2,3)$.
In the proof of Theorem \ref{main result 2}, we treat the small standing wave which bifurcates from the eigenfunction of $-\partial_x^2+V$ with respect to the lowest eigenvalue.
Since the line standing wave of the nonlinear Schr\"odinger equation studied in \cite{YY2} comes from the standing wave of the one dimensional nonlinear Schr\"odinger equation which has the scale invariant, we need to study the fully nonlinear structure of the Lyapunov functional around the line standing wave.
In this paper, using the smallness of the line standing wave of (\ref{NLS})  and the expansion of the standing wave with respect to the parameter $\omega$, we weaken the nonlinear structure of the Lyapunov functional around the line standing wave of (\ref{NLS}).
Therefore, we can evaluate a value of the integral and make a close investigation into the stability for all exponents $p \geq 2$.

The rest of this paper consists of the following three sections.
In Section 2, we show the properties of the spectrum and the coerciveness for the linearized operator around line standing waves.
In Section 3, applying the variational argument in \cite{G S S 1,C O} and the spectrum argument in \cite{YY1}, we prove Theorem \ref{main result 1}.
In Section 4, combining the bifurcation result and the argument for the degenerate case in \cite{MM}, we prove Theorem \ref{main result 2}.

\section{Preliminaries}
In this section, we investigate properties of the linearized operator of (\ref{NLS}) around the standing wave $e^{i\omega t}\tilde{\varphi}_{\omega}$.

Let $H^1(X)=\{ u :X \to \C | \int_X (|\nabla u|^2+|u|^2) dx < \infty \}$ and $H^1(X,\R)=\{ u :X \to \R| \int_X (|\nabla u|^2+|u|^2) dx < \infty\}$.
Let $\psi_{\omega}$ be the eigenfunction of $L_{\omega}^+$ corresponding to $-\lambda_{\omega}$ with $\norm{\psi_{\omega}}_{L^2(\R)}=1$ and $\psi_{\omega}>0$.
We define the action 
\[S_{\omega}(u) = E(u) + \omega Q(u).\]
Then, the action $S_{\omega}$ is a conservation law of (\ref{NLS}) and $S_{\omega}'(\tilde{\varphi}_{\omega})=0$, where $S_{\omega}'$ is the Fr\'echet derivation of $S_{\omega}$.
Moreover, we have
\[ \tbr{S''_{\omega}(\tilde{\varphi}_{\omega})u,v}_{H^{-1},H^1} = \tbr{\mathbb{L}_{\omega}^+ (\re u), \re v}_{H^{-1},H^1} + \tbr{\mathbb{L}_{\omega}^- (\im u), \im v}_{H^{-1},H^1},\]
where
\[
\mathbb{L}_{\omega}^+=-\Delta +\omega +V - p |\tilde{\varphi}_{\omega}|^{p-1},\quad
\mathbb{L}_{\omega}^-=-\Delta +\omega +V -  |\tilde{\varphi}_{\omega}|^{p-1}.
\]
Let 
\[ Ju=iu=
\begin{pmatrix}
0 &-1\\
1 &0
\end{pmatrix}
\begin{pmatrix}
\re u\\
\im u
\end{pmatrix} .\]
For $u,v \in L^2(\RTL)$, we define
\[\tbr{u,v}_{L^2}=\re \int_{\RTL} u \bar{v}dxdy.\] 
In the following proposition, we show properties of the spectrum of the linearized operator for (\ref{NLS}) around $\tilde{\varphi}_{\omega}$.
This proposition follows  Theorem 1.1 of \cite{R T 2} and Lemma 3.1 of \cite{R T 3} (also see Proposition 2.5 of \cite{YY1}.)
\begin{proposition}\label{eigenvalue}
Let $\lambda_*<\omega<\omega_*$.
\begin{enumerate}
  \setlength{\parskip}{0.1cm} 
  \setlength{\itemsep}{0cm} 
\renewcommand{\labelenumi}{\rm (\roman{enumi})}
\item If $0<L\leq (\lambda_{\omega})^{-1/2}$, then $-JS_{\omega}''(\tilde{\varphi}_{\omega})$ has no positive eigenvalue.
\item If $0<L< (\lambda_{\omega})^{-1/2}$, then
\[\mbox{\rm Ker}(S_{\omega}''(\tilde{\varphi}_{\omega}))=\mbox{\rm Span}\l\{i\tilde{\varphi}_{\omega}\r\}.\]
\item If $L=(\lambda_{\omega})^{-1/2}$, then
\[\mbox{\rm Ker}(S_{\omega}''(\tilde{\varphi}_{\omega}))=\mbox{\rm Span}\l\{i\tilde{\varphi}_{\omega},\psi_{\omega} \cos \frac{y}{L}, \psi_{\omega}\sin \frac{y}{L}\r\}.\]
\item If $L> (\lambda_{\omega})^{-1/2}$, then $-JS_{\omega}''(\tilde{\varphi}_{\omega})$ has a positive eigenvalue and the number of eigenvalue of $-JS_{\omega}''(\tilde{\varphi}_{\omega})$ with a positive real part is finite.
\end{enumerate}
Here, $\mbox{\rm Span} \{v_1, \dots, v_n\}$ is the real vector space spanned by vectors $v_1, \dots, v_n$.
\end{proposition}
\begin{proof}
We define
\[ S(a) u=
\begin{pmatrix}
L^+_{\omega}+a^2 & 0\\
0 & L^-_{\omega} + a^2
\end{pmatrix}
\begin{pmatrix}
\re u\\
\im u
\end{pmatrix},
\]
where $u \in H^1(\R)$, $L^+_{\omega}=-\partial_x^2 + \omega +V -p|\varphi_{\omega}|^{p-1}$ and $L^-_{\omega}=-\partial_x^2 + \omega +V -|\varphi_{\omega}|^{p-1}$.
Then, for $u \in H^1(\RTL)$
\[S_{\omega}''(\tilde{\varphi}_{\omega})u(x,y)=\sum_{n=-\infty}^{\infty} S\l( \frac{n}{L} \r)\vec{u}_n(x) e^{\frac{iny}{L}},\]
where
\[ u(x,y) =
\begin{pmatrix}
\re u\\
\im u
\end{pmatrix}
=\sum_{n=-\infty}^{\infty} 
\begin{pmatrix}
u_{R,n}\\
u_{I,n}
\end{pmatrix} e^{\frac{iny}{L}} = \sum_{n=-\infty}^{\infty} \vec{u}_n(x) e^{\frac{iny}{L}}, \quad u \in H^1(\RTL).\]
Therefore, $-JS''_{\omega}(\tilde{\varphi}_{\omega})$ has an eigenvalue $\lambda$ if and only if there exists $n \in \Z$ such that $-JS(n/L)$ has the eigenvalue $\lambda$.

By Proposition \ref{R-W}, $S(a)$ has no negative eigenvalues for $a\geq (\lambda_{\omega})^{1/2}$.
By Theorem 3.1 in \cite{P W}, the number of eigenvalues of $-JS(a)$ with the positive real part is less than or equal to the number of negative eigenvalues of $S(a)$.
Thus, for $a\geq (\lambda_{\omega})^{1/2}$, $JS(a)$ has no eigenvalues with the positive real part.
(i) follows this.
Moreover, the number of eigenvalues of $-JS''_{\omega}(\tilde{\varphi}_{\omega})$ with the positive real part is less than $1+2L/(\lambda_{\omega})^{1/2}$.
Since the kernel of $S(a)$ is trivial for $a> (\lambda_{\omega})^{1/2}$, the kernel of $-JS(a)$ is trivial for $a> (\lambda_{\omega})^{1/2}$.
Then the kernel of $-JS(0)$ is spanned by $i \tilde{\varphi}_{\omega}$.
Therefore, for $a> (\lambda_{\omega})^{1/2}$, the kernel of $-JS_{\omega}''(\tilde{\varphi})$ is spanned by $i \tilde{\varphi}_{\omega}$ and (ii) is verified.

The kernel of $L^+_{\omega}+\lambda_{\omega}$ is spanned by $\psi_{\omega}$ and $L^-_{\omega}+\lambda_{\omega}$ do not has zero eigenvalue.
Hence, the kernel of $S_{\omega}''(\tilde{\varphi}_{\omega})$ with $L=(\lambda_{\omega})^{-1/2}$ is spanned by $i\tilde{\varphi}_{\omega}$, $\psi_{\omega}\cos \frac{y}{L}$ and $\psi_{\omega} \sin \frac{y}{L}$.
This is (iii).

Let
\[ M(v,a,\lambda) = S(a) (\psi_{\omega} +v) +J^{-1}\lambda (\psi_{\omega}+v),\]
where $v \in \{ w \in H^2(\R)| \tbr{v,\psi_{\omega}}_{L^2(\R)}=0\}=:(\psi_{\omega})^\bot$ and $a, \lambda \in \R$.
Then, $M$ is a $C^{\infty}$ function with
\[ M(0, (\lambda_{\omega})^{1/2},0)=0.\]
Since
\[ \frac{\partial M}{\partial (v,a)}\biggl|_{(v,a,\lambda)=(0,(\lambda_{\omega})^{1/2},0)}(w,\mu) = 2(\lambda_{\omega})^{1/2}\mu \psi_{\omega} + S((\lambda_{\omega})^{1/2})w,\]
by the implicit function theorem, there exist $a(\lambda) \in \R$ and $v(\lambda) \in H^2(\R)$ such that $a(\lambda),v(\lambda)$ are the $C^{\infty}$ functions, where $a(0)=(\lambda_{\omega})^{1/2}$, $v(0)=0$ and $M(v(\lambda),a(\lambda),\lambda)=0$.
Then, we have
\[ -JS(a(\lambda))(\psi_{\omega}+v(\lambda))=\lambda (\psi_{\omega}+v(\lambda))\]
for sufficiently small $|\lambda|$.
Differentiating with respect to $\lambda$, we obtain
\[\frac{\partial M}{\partial \lambda} (v(\lambda),a(\lambda),\lambda) = 2a(\lambda)a'(\lambda) (\psi_{\omega} + v(\lambda))+S(a(\lambda))v'(\lambda) + J^{-1}(\psi_{\omega}+v(\lambda)+\lambda v(\lambda))=0.\]
Since $v(0)=0$, we have
\[\tbr{2a(0)a'(0)\psi_{\omega},\psi_{\omega}}_{L^2(\R)}=0,\]
and $a'(0)=0$.
Therefore, we have $S((\lambda_{\omega})^{1/2})v'(0)=-J^{-1}\psi_{\omega}$.
Since $v'(0) \in (\psi_{\omega})^{\bot}$ and 
\[\begin{split}
\frac{\partial^2 M}{\partial \lambda^2}(v(\lambda),a(\lambda),\lambda)|_{\lambda=0}
=2a(0)a''(0)\psi_{\omega}+S((\lambda_{\omega})^{1/2})v''(0) +J^{-1}v'(0)
=0,
\end{split}\]
\[a''(0)=-\frac{\tbr{S((\lambda_{\omega})^{1/2})v'(0),v'(0)}_{H^{-1}(\R),H^1(\R)}}{2(\lambda_{\omega})^{1/2}}<0.\]
From the proof of (i), for $a> (\lambda_{\omega})^{1/2} $, $-JS(a)$ has no positive eigenvalues.
Hence, for sufficiently small $\varepsilon >0$ the function $a(\lambda)$ on $(0,\varepsilon )$ has the inverse function $\lambda(a)$ on $(a(\varepsilon ), (\lambda_{\omega})^{1/2})$ and $a(\varepsilon )< (\lambda_{\omega})^{1/2}$.
Namely, $-JS(a)$ has the simple positive eigenvalue on $(a(\varepsilon ), (\lambda_{\omega})^{1/2})$.

Let 
\[ a_0 = \inf\{ a> 0| -JS(b) \mbox{ has a simple positive eigenvalue for } a< b <(\lambda_{\omega})^{1/2}  \},\]
and for $a \in (a_0,(\lambda_{\omega})^{1/2})$ the value $\lambda(a)$ be the positive eigenvalue of $-JS(a)$.
We assume $a_0>0$.
By the perturbation theory, there exists $\{a_n\}_{n=1}^{\infty} \subset (a_0,(\lambda_{\omega})^{1/2})$ such that $a_n \to a_0$ and 
\[\lim_{n \to \infty} \lambda(a_n) = 0\]
or
\[\lim_{n \to \infty}\lambda(a_n) = \infty.\]
Since there exists $C>0$ such that $|\tbr{-JS(a)u,u)}_{H^{-1}(\R),H^1(\R)}| \leq C\norm{u}_{L^2(\R)}$ for $a \in \R$, $\lambda(a)$ is bounded.
Therefore, 
\[\lim_{n \to \infty}\lambda(a_n) = 0.\]
Then, there exists $\{c_n\}_{n=1}^{\infty}$ such that $\norm{v_n}_{H^1(\R)}=1$ and $-JS(a_n)v_n=\lambda(a_n)v_n$.
Here, 
\[S(a_0)v_n = (a_0^2-a_n^2)v_n-J^{-1}\lambda(a_n)v_n.\]
Since $S(a_0)$ is invertible and $(S(a_0))^{-1}$ is bounded,
\[v_n=(S(a_0))^{-1} ((a_0^2-a_n^2)v_n - J^{-1}\lambda (a_n)v_n) \to 0 \mbox{ as } n \to \infty.\]
This is contradiction.
Therefore, $a_0=0$. 
\end{proof}

Next we show the coerciveness of $L_{\omega}^+$ on a function space which follows the proof of Theorem 3.3 in \cite{G S S 1}.
\begin{lemma}\label{L-est}
There exist $\omega_{*,0}>\lambda_*$ and $k_0>0$ such that for $\lambda_*< \omega <\omega_{*,0}$ and $u \in H^1(\R,\R) $ with $\tbr{\varphi_{\omega},u}_{L^2(\R)}=0$, 
\[\tbr{L^+_{\omega}u,u}_{H^{-1}(\R),H^1(\R)} \geq k_0 \norm{u}_{H^1(\R)}^2,\]
where $\tbr{u,v}_{L^2(\R)}=\re \int_{\R} u\bar{v} dx$.
\end{lemma}

\begin{proof}
Let $u \in H^1(\R,\R) $ with $\tbr{\varphi_{\omega},u}_{L^2(\R)}=0$.
We decompose $u=a\psi_* + u_{\bot}$, where $a=\tbr{u,\psi_*}_{L^2(\R)}$ and $\tbr{\psi_*,u_{\bot}}_{L^2(\R)}=0$. 
From the spectrum of $-\partial_x^2 +V+\lambda_*$, there exists $k>0$ such that 
\[\tbr{(-\partial_x^2 + V+\lambda_*)u_{\bot},u_{\bot}}_{H^{-1}(\R),H^1(\R)} \geq k \norm{u_{\bot}}_{L^2(\R)}^2.\]
By $\norm{V}_{L^{\infty}}<\infty$, we have for $\varepsilon >0$
\[\tbr{(-\partial_x^2 + V+\lambda_*)u_{\bot},u_{\bot}}_{H^{-1}(\R),H^1(\R)} \geq (k-\varepsilon (\norm{V}_{L^{\infty}}-\lambda_*)) \norm{u_{\bot}}_{L^2(\R)}^2 + \varepsilon \norm{\partial_x u_{\bot}}_{L^2(\R)}^2.\]
Therefore, there exists $k'>0$ such that
\[\tbr{(-\partial_x^2 + V+\lambda_*)u_{\bot},u_{\bot}}_{H^{-1}(\R),H^1(\R)} \geq k' \norm{u_{\bot}}_{H^1(\R)}^2.\]
By the assumption $\tbr{\varphi_{\omega},u}_{L^2(\R)}=0$, we have 
\[ a = -\frac{\tbr{\varphi_{\omega},u_{\bot}}_{L^2(\R)}}{\tbr{\varphi_{\omega},\psi_*}_{L^2(\R)}}.\]
Then 
\[ \begin{split}
\tbr{L_{\omega}^+ u,u}_{H^{-1}(\R),H^1(\R)}=&\tbr{(-\partial_x^2+V+\lambda_*)u_{\bot},u_{\bot}}_{H^{-1}(\R),H^1(\R)} +\lambda_{\omega} a^2 \\
&+ \tbr{(\omega-\lambda_* - p|\varphi_{\omega}|^{p-1} )u_{\bot},u_{\bot}}_{H^{-1}(\R),H^1(\R)} -2\lambda_{\omega}\frac{\tbr{\varphi_{\omega},u_{\bot}}_{L^2(\R)}^2}{\tbr{\varphi_{\omega},\psi_*}_{L^2(\R)}}+o(\norm{u_{\bot}}_{L^2(\R)}^2)\\
\geq & k' \norm{u_{\bot}}_{H^1}^2+\lambda_{\omega}a^2 + \tbr{(\omega-\lambda_* - p|\varphi_{\omega}|^{p-1} )u_{\bot},u_{\bot}}_{H^{-1}(\R),H^1(\R)} \\
&-2\lambda_{\omega}\frac{\norm{u_{\bot}}_{L^2(\R)}^2\norm{\varphi_{\omega}}_{L^2(\R)}^2}{\tbr{\varphi_{\omega},\psi_*}_{L^2(\R)}^2}+o(\norm{u_{\bot}}_{L^2(\R)}^2).\\
\end{split}\]
If $|\omega-\lambda_*|$ is sufficiently small, then we obtain the conclusion.
\end{proof}

\section{Proof of Theorem \ref{main result 1}}
In this section, we prove Theorem \ref{main result 1}.
The proof of Theorem \ref{main result 1} is similar to the proof of Theorem 1.5 in \cite{YY1}.
We write the detail of the proof of Theorem \ref{main result 1} for readers.

\subsection{Proof of (i) of Theorem \ref{main result 1}}

In this subsection, we assume $0<L<(\lambda_{\omega})^{-1/2}$.
The proof of (i) of Theorem \ref{main result 1} follows Section 3.1 in \cite{YY1}.

The following proposition follows Grillakis-Shatah-Strauss \cite{G S S 1} or Colin-Ohta \cite{C O}(see \cite{C C O 1}).
\begin{proposition}\label{p-3-1}
Let $e^{i\omega t} \varphi$ be a standing wave of (\ref{NLS}).
Assume that there exists a constant $\delta >0$ such that $\tbr{S_{\omega}''(\varphi)u,u}_{H^{-1},H^1} \geq \delta \norm{u}_{H^1}^2$ for all $u \in H^1(\RTL)$ satisfying $\tbr{\varphi,u}_{L^2}=\tbr{J\varphi,u}_{L^2}=0$.
Then, the standing wave $e^{i\omega t} \varphi$ is stable.
\end{proposition}

Let $u \in H^1(\RTL)$ satisfy $\tbr{\tilde{\varphi}_{\omega},u}_{L^2}=\tbr{J\tilde{\varphi}_{\omega},u}_{L^2}=0$.
Then, 
We have
\[ \tbr{S_{\omega}''(\tilde{\varphi}_{\omega})u,u}_{H^{-1},H^1} = \sum_{n \in \Z} \tbr{S(n/L)u_n,u_n}_{H^{-1},H^1},\]
where
\[ u(x,y) = \sum_{n \in \Z} u_n(x) e^{\frac{iny}{L}}.\]
Since $L_{\omega}^-$ and $L_{\omega}^+ + \lambda_{\omega}$ are nonnegative, there exists $c>0$ such that
\[\tbr{S(n/L)v,v}_{H^{-1}(\R),H^1(\R)} \geq c \norm{v}_{H^1(\R)}^2\]
for $n\in \Z\backslash \{0\}$ and $v \in H^1(\R)$.
By Proposition \ref{R-W} and Lemma \ref{L-est}, there exists $c'>0$ such that for $v \in H^1(\R)$ with $\tbr{\tilde{\varphi}_{\omega},v}_{L^2}=2\pi L \int_{\R} \varphi_{\omega} (\re v) dx =0$ and $\tbr{J\tilde{\varphi}_{\omega},v}_{L^2}=2\pi L\int_{\R} \varphi_{\omega} (\im v) dx =0$
\[ \tbr{\mathbb{L}^+_{\omega}(\re v),\re v}_{H^{-1},H^1} \geq c'\norm{\re v}_{H^1}^2,\] 
\[ \tbr{\mathbb{L}^-_{\omega}(\im v),\im v}_{H^{-1},H^1} \geq c'\norm{\im v}_{H^1}^2.\] 
Therefore, (i) of Theorem \ref{NLS} follows from Proposition \ref{p-3-1}.

\subsection{Proof of (ii) of Theorem \ref{main result 1}}
The proof of (ii) of Theorem \ref{main result 1} follows Section 3.2 in \cite{YY1}.

In this subsection, we assume $L>(\lambda_{\omega})^{1/2}$.
We define 
\[ \mu_* = \max\{ \lambda >0 | \lambda \in \sigma(-JS''(\tilde{\varphi}_{\omega}))\},\]
where $\sigma (-JS''(\tilde{\varphi}_{\omega}))$ is the spectrum of $-JS''(\tilde{\varphi}_{\omega})$.
Then, there exist $k_0 \in \Z$ and $\chi \in H^1(\RTL)$ such that $\norm{\chi}_{L^2}=1$,  $\chi$ is eigenfunction of $-JS''_{\omega}(\tilde{\varphi}_{\omega})$ corresponding to $\mu_*$ and
\[ \chi(x,y) = \chi_1(x) e^{\frac{ik_0y}{L}} + \chi_2(x) e^{\frac{-ik_0y}{L}},\]
where $\chi_1, \chi_2 \in H^2(\R)$.
We define the orthogonal projection $P_{\leq k}$ as
\[P_{\leq k} u(x,y) = \sum_{n=-k}^k u_n(x) e^{\frac{iny}{L}}, \quad (x,y) \in \RTL,\]
where
\[u(x,y) = \sum_{n=-\infty}^{\infty} u_n(x)e^{\frac{iny}{L}}. \quad (x,y) \in \RTL.\]
A function $u(t)$ is a solution of (\ref{NLS}) if and only if $v(t)$ is a solution of the equation
\begin{equation}\label{NLS2}
\partial_t v = -J(S_{\omega}''(\tilde{\varphi}_{\omega}) v + g(v)),
\end{equation}
where $u(t)=e^{i\omega t}(\tilde{\varphi}_{\omega}+v(t))$,
\[ g(v)=
\begin{pmatrix}
|v+\tilde{\varphi}_{\omega}|^{p-1}(v_R+\tilde{\varphi}_{\omega}) -p|\tilde{\varphi}_{\omega}|^{p-1}v_R -|\tilde{\varphi}_{\omega}|^{p-1}\tilde{\varphi}_{\omega}\\
|v+\tilde{\varphi}_{\omega}|^{p-1}v_I -|\tilde{\varphi}_{\omega}|^{p-1}v_I
\end{pmatrix}
,
\]
and $v_R= \re v$ and $v_I=\im v$.
We define $u_{\delta}(t)$ as the solution of (\ref{NLS}) with the initial data $\tilde{\varphi}_{\omega}+\delta\chi$ and $v_{\delta}(t)$ as the solution of (\ref{NLS2}) with the data $\delta \chi$.
Then, we have that $u_{\delta}(t)=e^{i\omega t}(\tilde{\varphi}_{\omega}+v_{\delta}(t))$.

We show the estimate of nonlinear term in the following lemma which follows Lemma 2.4 of \cite{G V 2}.
\begin{lemma}\label{lem-3-1}
There exists $C>0$ such that
\[\norm{g(v)}_{L^2}\leq 
\begin{cases}
C\norm{v}_{H^1}^{p}, \quad 1<p\leq 2,\\
C(\norm{v}_{H^1}^2+\norm{v}_{H^1}^{p}), \quad 2<p.
\end{cases}\]
\end{lemma}
\begin{proof}
We have
\[
||a|^{p-1}-|b|^{p-1}|\leq 
\begin{cases}
|a-b|^{p-1}, \quad 1<p\leq 2,\\
p(|a|^{p-2}+|b|^{p-2})|a-b|, \quad 2<p.
\end{cases}
\]
Since
\[g(v(x,y))=\int_0^1 (|\theta v(x,y)+\tilde{\varphi}_{\omega}(x,y)|^{p-1}-|\tilde{\varphi}_{\omega}(x,y)|^{p-1})v(x,y)d\theta,\]
\[\begin{split}
\norm{g(v)}_{L^2} & \leq \norm{\int_0^1 (|\theta v+\tilde{\varphi}_{\omega}|^{p-1}-|\tilde{\varphi}_{\omega}|^{p-1})vd\theta}_{L^2}\\
&\leq \int_0^1 \norm{|\theta v+\tilde{\varphi}_{\omega}|^{p-1}-|\tilde{\varphi}_{\omega}|^{p-1}}_{L^4}\norm{v}_{L^4}d\theta\\
& \leq 
\begin{cases}
C \norm{v}_{H^1}^{p},\quad 1<p\leq 2,\\
C (\norm{v}_{H^1}^2+\norm{v}_{H^1}^p),\quad 2<p.
\end{cases}
\end{split}
\]
\end{proof}

In the following lemma, we estimate the low frequency part of the semi-group.
\begin{lemma}\label{lem-3-2}
For a positive integer $k$ and $\varepsilon >0$, there exists $C_{k,\varepsilon  }>0$ such that
\[\norm{e^{-tJS''_{\omega}(\tilde{\varphi}_{\omega})}P_{\leq k}v}_{L^2} \leq C_{k,\varepsilon } e^{(\mu_*+\varepsilon )t}\norm{v}_{L^2}, \quad t>0, v \in L^2(\RTL).\]
\end{lemma}
\begin{proof}
By the definition of $S(a)$, we have
\[-JS(a)=
\begin{pmatrix}
0 & -\partial_x^2 + \omega+a^2 +V -|\tilde{\varphi}_{\omega}|^{p-1}\\
\partial_x^2-\omega-a^2-V+p|\tilde{\varphi}_{\omega}|^{p-1} & 0
\end{pmatrix}
.
\]
Using the exponential decay rates of $V$ and $\tilde{\varphi}_{\omega}$ and applying the argument for the proof of Proposition \cite{G O} and Lemma 6 in \cite{G J L S}, we obtain 
\[\sigma(e^{-JS(a)})=e^{\sigma(-JS(a))}.\]
By the definition of $\mu_*$, we have that the spectral radius of $e^{-JS(n/L)}$ is less than or equal to $e^{\mu_*}$ for $n \in \Z$.
Therefore, by Lemma 2 and Lemma 3 in \cite{S S} we have
\[\norm{e^{-tJS(n/L)}v}_{L^2(\R)} \leq C_{n,\varepsilon } e^{(\mu_*+\varepsilon )t} \norm{v}_{L^2(\R)}, \quad t>0,n \in \Z,v \in L^2(\R).\]
Hence, for $t>0$ and $v \in L^2(\RTL)$,
\[\norm{e^{-tJS_{\omega}''(\tilde{\varphi}_{\omega})}P_{\leq k}v}_{L^2} \leq \norm{\sum_{n=-k}^k e^{-tJS(n/L)}v_ne^{\frac{iny}{L}}}_{L^2}\leq C_{k,\varepsilon }e^{(\mu_* +\varepsilon )t} \norm{v}_{L^2},\]
where
\[v(x,y)=\sum_{n\in \Z} v_n(x)e^{\frac{iny}{L}}.\]
\end{proof}

In the following lemma, we estimate the high frequency part of $v_{\delta}(t)$.
\begin{lemma}\label{lem-3-3}
There exist a positive integer $K_0$ and $C>0$ such that for $\delta>0$ and $t>0$
\[\norm{v_{\delta}(t)}_{H^1} \leq C\norm{P_{\leq K_0}v_{\delta}(t)}_{L^2} + o(\delta)+o(\norm{v_{\delta}(t)}_{H^1}).\]
\end{lemma}
\begin{proof}
By the Taylor expansion we have that for $v \in H^1(\RTL)$
\[S_{\omega}(\tilde{\varphi}_{\omega}+v)=S_{\omega}(\tilde{\varphi}_{\omega}) + \tbr{S_{\omega}'(\tilde{\varphi}_{\omega}),v}_{H^{-1},H^1} + \frac{1}{2}\tbr{S_{\omega}''(\tilde{\varphi}_{\omega})v,v}_{H^{-1},H^1}+o(\norm{v}_{H^1}^2).
\]
Since $S_{\omega}$ is conservation law, we have $S_{\omega}(\tilde{\varphi}_{\omega}+\delta \chi)=S_{\omega}(\tilde{\varphi}_{\omega}+v_{\delta}(t))$ for $t\geq 0$.
Using $S'(\tilde{\varphi}_{\omega})=0$ and
\[\tbr{S_{\omega}''(\tilde{\varphi}_{\omega})\chi,\chi}_{H^{-1},H^1}=\tbr{-JS''_{\omega}(\tilde{\varphi}_{\omega})\chi,J^{-1}\chi}_{H^{-1},H^1}=\tbr{\mu_* \chi, J^{-1}\chi}_{L^2}=0,\]
we have
\[\tbr{S_{\omega}''(\tilde{\varphi}_{\omega})v_{\delta}(t),v_{\delta}(t)}_{H^{-1},H^1} = o(\norm{v_{\delta}(t)}_{H^1}^2)+o(\delta^2).\]
We define $K_0$ as the integer part of $1+L(\lambda_{\omega})^{1/2}$.
Since $S(a)$ is positive for $a>(\lambda_{\omega})^{1/2}$, we obtain $S_{\omega}''(\tilde{\varphi}_{\omega})(I-P_{\leq K_0})$ is positive.
By the definition of $S(a)$ there exist $c,C>0$ such that
\[ \tbr{S(a)v,v}_{H^{-1}(\R),H^1(\R)} \geq c\norm{v}_{H^1(\R)}^2-C\norm{v}_{L^2(\R)}^2\]
for $v \in H^1(\R)$ and $a \in \R$.
Thus, 
\[ \begin{split}
\norm{v_{\delta}(t)}_{H^1}^2&=\norm{P_{\leq K_0}v_{\delta}(t)}_{H^1}^2 + \norm{(I-P_{\leq K_0})v_{\delta}(t)}_{H^1}^2\\
&\leq C'\tbr{S_{\omega}''(\tilde{\varphi}_{\omega})(I-P_{\leq K_0})v_{\delta}(t),(I-P_{\leq K_0})v_{\delta}(t)}_{H^{-1},H^1} \\
& \quad + C'\tbr{S_{\omega}''(\tilde{\varphi}_{\omega})P_{\leq K_0}v_{\delta}(t),P_{\leq K_0}v_{\delta}(t)}_{H^{-1},H^1} + C''\norm{P_{\leq K_0}v_{\delta}(t)}_{H^1}^2\\
&\leq C''\norm{v_{\delta}(t)}_{L^2}^2 + o(\delta^2) + o(\norm{v_{\delta}(t)}_{H^1}^2).
\end{split}\] 
\end{proof}

Let $\varepsilon  _0 =\min\{(p-1)\mu_*/2,\mu_*/2\}$.
By Lemma \ref{lem-3-1}, Lemma \ref{lem-3-2} and Lemma \ref{lem-3-3}, we obtain that
\[\begin{split}
\norm{v_{\delta}(t)}_{H^1} &\leq C \delta e^{\mu_*t}\norm{\chi}_{L^2} + C\int_0^t\norm{e^{-(t-s)JS_{\omega}(\tilde{\varphi}_{\omega})}P_{\leq K_0}g(v_{\delta}(s))}_{L^2} ds + o(\delta) + o(\norm{v_{\delta}(t)}_{H^1})\\
& \leq C\delta e^{\mu_* t} + C \int_0^t e^{(1+\varepsilon _0)\mu_*(t-s)}(\norm{v_{\delta}(s)}_{H^1}^2+\norm{v_{\delta}(s)}_{H^1}^p)ds+o(\delta) + o(\norm{v_{\delta}(t)}_{H^1}).
\end{split}\]
There exists $C_0>0$ such that for small $\delta>0$ and $\varepsilon_1 >0$ 
\[\norm{v_{\delta}(t)}_{H^1} \leq C_0 e^{\mu_* t}, \quad \mbox{ for } t \in [0,T_{\varepsilon_1 ,\delta}],\]
where
\[T_{\varepsilon _1,\delta}=\frac{\log (\varepsilon _1/\delta)}{\mu_*}.\]
Then,
\[ \begin{split}
|\tbr{\chi,v_{\delta}(T_{\varepsilon _1,\delta})}_{L^2}|&=\l| \delta e^{\mu_* T_{\varepsilon _1, \delta}}+\int_0^{T_{\varepsilon _1,\delta}}\tbr{\chi,-Je^{-(T_{\varepsilon _1,\delta}-s)S_{\omega}''(\tilde{\varphi}_{\omega})}g(v_{\delta}(s)}_{L^2}ds \r|\\
&\geq \varepsilon _1 -C \int_0^{T_{\varepsilon _1,\delta}} e^{(T_{\varepsilon _1,\delta}-s)\mu_*}(\norm{v_{\delta}(s)}_{H^1}^2+\norm{v_{\delta}(s)}_{H^1}^p)ds\\
& \geq \varepsilon _1 - C \varepsilon_1 ^{\min \{p,2\}}.
\end{split}\]
Since $P_{\leq 0} \tilde{\varphi}_{\omega}=\tilde{\varphi}_{\omega}$,
there exists $\varepsilon _1 >0$ such that for $\varepsilon _1 > \delta >0$ and $\theta \in \R$
\[\begin{split}
\norm{u_{\delta}(T_{\varepsilon _1,\delta})-e^{i\theta}\tilde{\varphi}_{\omega}}_{L^2}
&\geq \norm{(I-P_{\leq 0}) (u_{\delta}(T_{\varepsilon _1,\delta})-e^{i\theta}\tilde{\varphi}_{\omega})}_{L^2}\\
&= \norm{(I-P_{\leq 0})e^{-i\omega T_{\varepsilon _1,\delta}}u_{\delta}(T_{\varepsilon _1,\delta})}_{L^2}\\
&= \norm{(I-P_{\leq 0})(e^{-i\omega T_{\varepsilon _1,\delta}}u_{\delta}(T_{\varepsilon _1,\delta})-\tilde{\varphi}_{\omega})}_{L^2}.
\end{split}\]
By the definition of $\chi$ we have
\[\norm{(P_{\leq k_0}-P_{\leq k_0-1})v}_{L^2} \geq |\tbr{\chi,v}_{L^2}|, \quad \mbox{ for } v \in L^2(\RTL).\]
Therefore, 
\[\norm{(I-P_{\leq 0})(e^{-i\omega T_{\varepsilon _1,\delta}}u_{\delta}(T_{\varepsilon _1,\delta})-\tilde{\varphi}_{\omega})}_{L^2} \geq |\tbr{\chi,v_{\delta}(T_{\varepsilon _1,\delta})}_{L^2}| \geq \frac{\varepsilon _1}{2}.\]
This implies that the standing wave $e^{i\omega t}\tilde{\varphi}_{\omega}$ is unstable.

\section{Proof of Theorem \ref{main result 2}}

In this section, we assume $L=(\lambda_{\omega_0})^{-1/2}$ for $0<\omega_0 -\lambda_* \ll 1$.

The following lemma follows Proposition \ref{R-W}.

\begin{lemma}\label{l-4-1}
\begin{equation}
\begin{split}
\varphi_{\omega}&=\norm{\psi_*}_{L^{p+1}(\R)}^{-\frac{p+1}{p-1}}(\omega-\lambda_*)^{\frac{1}{p-1}} \psi_* + \norm{\psi_*}_{L^{p+1}(\R)}^{-\frac{p+1}{p-1}} (\omega-\lambda_*)^{\frac{1}{p-1}+1}(P_{\bot}^1H_{\lambda_*}P_{\bot}^1)^{-1}\psi_{*,p} \\
& \quad +o((\omega-\lambda_*)^{\frac{1}{p-1}+1}),\\
\partial_{\omega}\varphi_{\omega}&=\frac{\norm{\psi_*}_{L^{p+1}(\R)}^{-\frac{p+1}{p-1}}}{p-1}(\omega-\lambda_*)^{\frac{1}{p-1}-1} \psi_* + \frac{p\norm{\psi_*}_{L^{p+1}(\R)}^{-\frac{p+1}{p-1}}}{p-1} (\omega-\lambda_*)^{\frac{1}{p-1}}(P_{\bot}^1H_{\lambda_*}P_{\bot}^1)^{-1}\psi_{*,p} \\
&\quad +o((\omega-\lambda_*)^{\frac{1}{p-1}}).\\
\end{split}
\end{equation}
where $\psi_{*,p}= \norm{\psi_*}_{L^{p+1}(\R)}^{-(p+1)} \psi_*^p-\psi_*$,  $H_a=-\partial_x^2 + a +V$ for $a \in \R$ and $P_{\bot}^1$ is the orthogonal projection onto $(\psi_*)^{\bot}=\{ u \in L^2(\R) | \tbr{u,\psi_*}_{L^2(\R)}=0\}$.
\end{lemma}

\begin{proof}
Let
\[ \varphi_{\omega,0}=(\omega-\lambda_*)^{-\frac{1}{p-1}}\varphi_{\omega}=\norm{\psi_*}_{L^{p+1}(\R)}^{-\frac{p+1}{p-1}}\psi_* + \tilde{r}.\]
By the bifurcation argument, $\tbr{\psi_*,\tilde{r}(\omega)}_{L^2(\R)}=0$.
Since $\varphi_{\omega}$ is $C^1$ with respect to $\omega$ and
\[ (-\partial_x^2 + \omega +V) \varphi_{\omega,0} - |\varphi_{\omega}|^{p-1}\varphi_{\omega,0} =0, \]
we have
\[ \begin{split}
0&=\partial_{\omega}\bigl( (-\partial_x^2 + \omega +V) \varphi_{\omega,0} - |\varphi_{\omega}|^{p-1}\varphi_{\omega,0} \bigr)\\
&=\norm{\psi_*}_{L^{p+1}(\R)}^{-\frac{p+1}{p-1}} \psi_* + \tilde{r} + H_{\omega} \partial_{\omega} \tilde{r} -\l|\norm{\psi_*}_{L^{p+1}(\R)}^{-\frac{p+1}{p-1}} \psi_* + \tilde{r}  \r|^{p-1}\l(\norm{\psi_*}_{L^{p+1}(\R)}^{-\frac{p+1}{p-1}}\psi_* + \tilde{r}\r) \\
& \quad -p(\omega -\lambda_*)\l|\norm{\psi_*}_{L^{p+1}(\R)}^{-\frac{p+1}{p-1}} \psi_* + \tilde{r}  \r|^{p-1} \partial_{\omega} \tilde{r}.
\end{split} \]
From $\tbr{\psi_*,\tilde{r}(\omega)}_{L^2(\R)}=\tbr{\psi_*,\norm{\psi_*}_{L^{p+1}(\R)}^{-(p+1)} \psi_*^p-\psi_*}_{L^2(\R)}=0$, we obtain 
\[ \partial_{\omega} \tilde{r} = \norm{\psi_*}_{L^{p+1}(\R)}^{-\frac{p+1}{p-1}}( P_{\bot}^1 H_{\lambda_*} P_{\bot}^1)^{-1} (\norm{\psi_*}_{L^{p+1}(\R)}^{-(p+1)} \psi_*^p-\psi_*) +o(1).\]
\end{proof}

In the following lemma, we obtain the derivative of the eigenvalue $\lambda_{\omega}$.
\begin{lemma}\label{l-4-2}
Let $p\geq 2$.
Then,
\begin{equation*}
\psi_{\omega}= \psi_*+ p(\omega-\lambda_*)(P_{\bot}^1H_{\lambda_*}P_{\bot}^1)^{-1}\psi_{*,p} +O((\omega-\lambda_*)^2),
\end{equation*}
\begin{equation*}
\begin{split}
\lambda_{\omega}=&(p-1)(\omega-\lambda_*)+p(2p-1)(\omega-\lambda_*)^2\norm{\psi_*}_{L^{p+1}(\R)}^{-(p+1)}\int_{\R} \psi_*^p (P_{\bot}^1 H_{\lambda_*} P_{\bot}^1)^{-1} \psi_{*,p}dx \\
&+ o((\omega-\lambda_*)^2).
\end{split}
\end{equation*}

\end{lemma}

\begin{proof}
There exists $\delta_0>0$ such that $\{z\in\C | |z+\lambda_*|<2\delta_0\} \cap \sigma(-\Delta + V)=\{-\lambda_*\}$.
Let $\Gamma=\{z \in \C| |z|=\delta_0 \}$ be a simple closed curve and projections
\[ P_{\omega}= \frac{1}{2\pi i}\int_{\Gamma} (L^+_{\omega}-z)^{-1} dz.\]
Then, for $\omega>\lambda_*$ with $0<\omega-\lambda_*\ll 1$, 
\[ P_{\omega}u=\tbr{u,\psi_{\omega}}_{L^2(\R)}\psi_{\omega}.\]
Since $p\geq 2$, $L^+_{\omega}$ is $C^1$ with respect to $\omega$.
Therefore, the projection $P_{\omega}$ is also $C^1$.
For $\omega, \omega'>\lambda_*$, $(\tbr{\psi_{\omega'},\psi_{\omega}}_{L^2(\R)})^2 -1=\tbr{P_{\omega'}\psi_{\omega},\psi_{\omega}}_{L^2(\R)}-1=o(1)$ as $|\omega'-\omega| \to 0$.
For $\omega>\lambda_*$, 
\[\begin{split}
\psi_{\omega'}-\psi_{\omega}=\frac{P_\omega(\psi_{\omega}+\psi_{\omega'})-P_{\omega'}(\psi_{\omega}+\psi_{\omega'})}{1+\tbr{\psi_{\omega'},\psi_{\omega}}_{L^2(\R)}}.
\end{split}\]
Thus, $\psi_{\omega}$ is $C^1$ with respect to $\omega$.
Let $\varphi_{\omega,0}=(\omega-\lambda_*)^{-\frac{1}{p-1}}\varphi_{\omega}$.
Since $L^+_{\omega}\psi_{\omega}=-\lambda_{\omega}\psi_{\omega}$, we have
\[ -\lambda_{\omega}=\tbr{L_{\omega}^+\psi_{\omega},\psi_{\omega}}_{H^{-1}(\R),H^1(\R)}.\]
Therefore, 
\begin{equation}\label{l-4-2-1} 
\begin{split}
-\frac{d}{d\omega}\lambda_{\omega}
&= 1-p \int_{\R}(\varphi_{\omega,0})^{p-1}(\psi_{\omega})^2dx -p \int_{\R}(p-1)(\omega-\lambda_*)(\varphi_{\omega,0})^{p-2}(\partial_{\omega}\varphi_{\omega,0})(\psi_{\omega})^2dx\\
&= 1-p + O(\omega-\lambda_*).
\end{split}
\end{equation}
Since
\[ (-\partial_x^2 + \omega + \lambda_{\omega} +V -p|\varphi_{\omega}|^{p-1}) \psi_{\omega}=0,\]
we have
\[ \begin{split}
0
&=(1+\partial_{\omega}\lambda_{\omega} -p|\varphi_{\omega,0}|^{p-1}-p(p-1)(\omega-\lambda_*)|\varphi_{\omega,0}|^{p-2}\partial_{\omega}\varphi_{\omega,0}) \psi_{\omega} + L_{\omega}^+ \partial_{\omega} \psi_{\omega}.\\
\end{split}\]
Therefore,
\[ \begin{split}
 \partial_{\omega} \psi_{\omega} 
&=p(P_{\bot}^1H_{\lambda_*} P_{\bot}^1)^{-1}(\norm{\psi_*}_{L^{p+1}(\R)}^{-(p+1)}\psi_*^p-\psi_*)+ O(\omega-\lambda_*).
\end{split}\]
By (\ref{l-4-2-1}) and lemma  \ref{l-4-1}, we obtain 
\[\begin{split}
\frac{d^2}{d\omega^2}\lambda_{\omega}=&2p(p-1) \int_{\R}(\varphi_{\omega,0})^{p-2}(\partial_{\omega}\varphi_{\omega,0})(\psi_{\omega})^2dx +2p \int_{\R}(\varphi_{\omega,0})^{p-1}\psi_{\omega}\partial_{\omega} \psi_{\omega}dx \\
&+p(p-1) \int_{\R}(\omega-\lambda_*)(\varphi_{\omega,0})^{p-3}(\partial_{\omega}\varphi_{\omega,0})^2(\psi_{\omega})^2dx \\
&+p(p-1) \int_{\R}(\omega-\lambda_*)(\varphi_{\omega,0})^{p-2}(\partial_{\omega}^2\varphi_{\omega,0})(\psi_{\omega})^2dx \\
& +2p(p-1) \int_{\R}(\omega-\lambda_*)(\varphi_{\omega,0})^{p-2}(\partial_{\omega}\varphi_{\omega,0})\psi_{\omega}\partial_{\omega} \psi_{\omega}dx\\
=& 2p(2p-1)\int_{\R} \norm{\psi_*}_{L^{p+1}(\R)}^{-(p+1)}\psi_*^p(P_{\bot}^1H_{\lambda_*}P_{\bot}^1)^{-1}(\norm{\psi_*}_{L^{p+1}(\R)}^{-(p+1)}\psi_*^p-\psi_*) dx + o(1).\\
\end{split}\]
\end{proof}

The following corollary follows Lemma \ref{l-4-2}.
\begin{corollary}
There exists $\omega_{*,1}>\lambda_*$ such that for $\lambda_* < \omega<\omega_{*,1}$, $\lambda_{\omega}>0$.
Moreover, if $\lambda_*<\omega_0<\omega_{*,1}$, then the followings are hold.
\begin{enumerate}
  \setlength{\parskip}{0.1cm} 
  \setlength{\itemsep}{0cm} 
\renewcommand{\labelenumi}{\rm (\roman{enumi})}
\item If $\omega_0<\omega<\omega_{*,1}$, then $\mathbb{L}_{\omega}^+$ has exactly two negative eigenvalue and no kernel.
\item If $\lambda_* < \omega<\omega_0$, then $\mathbb{L}_{\omega}^+$ has exactly one negative eigenvalue and no kernel.
\end{enumerate}
\end{corollary}

Applying Lyapunov-Schmidt decomposition and Crandall-Rabinowitz Transversality in \cite{K K P}, we show $\tilde{\varphi}_{\omega_0}$ is a bifurcation point.
In this paper, we only write the sketch of the proof of the following proposition(see the proof of Theorem 4 in \cite{K K P} or Proposition 1 in \cite{YY2} for the detail of the proof of the following proposition).
\begin{proposition}\label{bifurcation}
Let $p\geq 2$ and $\lambda_*<\omega_0<\omega_{*,1}$.
There exist $\delta>0$ and $\phi_{\omega_0}\in C^2([-\delta,\delta],H^2)$ such that $\phi_{\omega_0}(a)>0$, 
\[\phi_{\omega_0}(a)(x,y)=\phi_{\omega_0}(a)(-x,y)=\phi_{\omega_0}(a)(x,-y),\quad (x,y) \in \R \times [-\pi L,\pi L],\]
\[-\Delta \phi_{\omega_0}(a) + \omega_{\omega_0}(a) \phi_{\omega_0}(a) +V \phi_{\omega_0}(a)-|\phi_{\omega_0}(a)|^{p-1}\phi_{\omega_0}(a)=0,\]
\[\phi_{\omega_0}(a)=\tilde{\varphi}_{\omega_0}+a\psi_{\omega_0}\cos \frac{y}{L} + r_{\omega_0}(a),\]
\begin{equation}\label{omega-ex}
 \omega_{\omega_0}(a) = \omega_0+\frac{\omega''_{\omega_0}(0)}{2}a^2 + o(a^2),
\end{equation}
where $r_{\omega_0}(a) \bot \psi_{\omega_0}\cos \frac{y}{L}$, $\norm{r_{\omega_0}(a)}_{H^2}=O(a^2)$,
\begin{equation}\label{omega-val}
\begin{split}
\omega_{\omega_0}''(0)&=\frac{-p^2(p-1)^2}{\frac{d\lambda_{\omega}}{d\omega}|_{\omega=\omega_0}\norm{\psi_{\omega_0}\cos \frac{y}{L}}_{L^2}^2}
\tbr{(\tilde{\varphi}_{\omega_0})^{p-2}(\psi_{\omega_0}\cos \frac{y}{L})^2,\mathbb{L}_{\omega_0}^{-1}((\tilde{\varphi}_{\omega_0})^{p-2}(\psi_{\omega_0}\cos \frac{y}{L})^2)}_{L^2}\\
& \quad - \frac{p(p-1)(p-2)}{3\frac{d\lambda_{\omega}}{d\omega}|_{\omega=\omega_0}\norm{\psi_{\omega_0}\cos \frac{y}{L}}_{L^2}^2}\tbr{(\psi_{\omega_0}\cos \frac{y}{L})^2,(\tilde{\varphi}_{\omega_0})^{p-3}(\psi_{\omega_0}\cos \frac{y}{L})^2}_{L^2},
\end{split}
\end{equation}
\begin{equation}\label{eigen-val}
 \lambda_2(a)=\frac{d\lambda_{\omega}}{d\omega}|_{\omega=\omega_0} \omega_{\omega_0}''(0)a^2 + o(a^2),
\end{equation}
and
\begin{equation}\label{norm-val}
\norm{\phi_{\omega_0}(a)}_{L^2}^2=\norm{\tilde{\varphi}_{\omega_0}}_{L^2}^2+ \frac{R_{p,\omega_0}}{2} a^2+o(a^2).
\end{equation}
Here,
\[R_{p,\omega_0}=-2\frac{d\lambda_{\omega}}{d\omega}|_{\omega=\omega_0} \norm{\psi_{\omega_0} \cos \frac{y}{L}}_{L^2}^2 + \omega_{\omega_0}''(0)\frac{d\norm{\tilde{\varphi}_{\omega}}_{L^2}^2}{d\omega}\biggl|_{\omega=\omega_0},\] 
 and $\lambda_2(a)$ is the second eigenvalue of $\mathbb{L}(a,\omega_0)=-\Delta +\omega_{\omega_0}(a)+ V - |\phi_{\omega_0}(a)|^{p-1}$.
\end{proposition}

\begin{proof}[The sketch of the proof]
Let $F$ be the function from $H^2_{sym}(\RTL,\R) \to L^2_{sym}(\RTL,\R)$ satisfying
\[F(\varphi,\omega)=-\Delta \varphi +\omega \varphi +V\varphi -|\varphi|^{p-1}\varphi,\]
where $L^2_{sym}(\RTL,\R)=\{u \in L^2(\RTL,\R)|u(x,y)=u(-x,y)=u(x,-y), (x,y) \in \R\times [-\pi L,\pi L]\}$, $H^2_{sym}(\RTL,\R)=H^2(\RTL)\cap L^2_{sym}(\RTL,\R)$ and $L^2(\RTL,\R)$ is the set of real valued $L^2$-function on $\RTL$.
Then, $\mbox{\rm Ker}(\partial_{\varphi}F(\tilde{\varphi}_{\omega_0},\omega_0))$ is spanned by $\psi_{\omega_0}\cos \frac{y}{L}$.
Applying the Lyapunov-Schmidt decomposition, we obtain that there exists a function $h(\omega,a) \in H^2_{sym}(\RTL,\R)$ such that 
\[P_{\bot} F(\tilde{\varphi}_{\omega_0}+a \psi_{\omega_0}\cos \frac{y}{L} + h(\omega,a),\omega)=0,\]
where $P_{\bot}$ is the orthogonal projection onto $\{u \in L^2(\RTL,\R)|\tbr{u,\psi_{\omega_0}\cos \frac{y}{L}}_{L^2}=0\}$.
Then, the problem $F(\tilde{\varphi}_{\omega_0}+a \psi_{\omega_0}\cos \frac{y}{L} + h(\omega,a),\omega)=0$ is equivalent to the problem  
\[ F_{||}(\omega,a)=\tbr{F(\tilde{\varphi}_{\omega_0}+a \psi_{\omega_0}\cos \frac{y}{L} + h(\omega,a),\omega), \psi_{\omega_0} \cos \frac{y}{L}}_{L^2}=0.\]
We apply the Crandall-Rabinowitz Transversality and we consider the problem $g(\omega,a)=0$, where 
\[ g(\omega,a)=
\begin{cases}
\frac{F_{||}(\omega,a)-F_{||}(\omega,0)}{a}, \quad a \neq 0,\\
\frac{\partial F_{||}}{\partial a}(\omega,0), \quad a =0.
\end{cases}
\]
Here for $a \neq 0$, $F_{||}(\omega,a)=0$ if and only if $g(\omega,a)=0$.
If $p>2$, then $F_{||}$ is a $C^2$ function and $g$ is a $C^1$ function.
In the case $p=2$, by the positivity of $\tilde{\varphi}_{\omega_0}$ and the Lebesgue dominant converge theorem, we can prove $g$ is $C^1$.
Then, 
\[\frac{\partial g}{\partial \omega}(\omega_0,0)=\frac{\partial \lambda_{\omega}}{\partial \omega}\biggl|_{\omega=\omega_0}\norm{\psi_{\omega_0}\cos \frac{y}{L}}_{L^2}^2, \quad \frac{\partial g}{\partial a}(\omega_0,0)=0.\]
Therefore, by the implicit function theorem there exists $\omega_{\omega_0}(a)$ such that $g(\omega_{\omega_0}(a),a)=0$.
Hence, $\phi_{\omega_0}(a):=\tilde{\varphi}_{\omega_0} + a \psi_{\omega_0} \cos \frac{y}{L} + h(\omega_{\omega_0}(a),a)$ is a solution of $F(\phi_{\omega_0}(a),\omega_{\omega_0}(a))=0$ and
\[\omega_{\omega_0}'(0)=-\frac{\frac{\partial g}{\partial a}}{\frac{\partial g}{\partial \omega}}(\omega_0,0) =0.\]
Using certain upper and lower exponential decay rates and positivity of $\phi_{\omega_0}(a)$, we can obtain
\[\omega_{\omega_0}''(0) = \lim_{a \to 0}\frac{\omega_{\omega_0}'(0)}{a} = \frac{-1}{\frac{\partial \lambda_{\omega}}{\partial \omega}|_{\omega=\omega_0}} \lim_{a \to 0} \frac{1}{a}\frac{\partial g}{\partial a}(\omega_{\omega_0}(a),a),\]
and (\ref{omega-val}).

Since $\mathbb{L}(a,\omega_0)$ is $C^1$, there exists an eigenfunction $\chi_*(a)$ of $\mathbb{L}(a,\omega_0)$ corresponding to $\lambda_2(a)$ such that $\norm{\chi_*(a)}_{L^2}=1$, $\chi_*(0)=\norm{\psi_{\omega_0} \cos \frac{y}{L}}_{L^2}^{-1}\psi_{\omega_0}\cos \frac{y}{L}$ and $\chi_*(a)$ is $C^1$ with respect to $a$.
In the case $p>2$, since $F_{||}$ is $C^2$, $\phi_{\omega_0}(a)$ is $C^2$.
In the case $p=2$, since $\mathbb{L}(a,\omega_0)$ is $C^1$ and
\[\frac{d\phi_{\omega_0}}{da}(a)=\psi_{\omega_0}\cos \frac{y}{L} - (P_{\bot}\mathbb{L}(a,\omega_0)P_{\bot})^{-1}P_{\bot}(\mathbb{L}(a,\omega_0)\psi_{\omega_0}\cos \frac{y}{L} + \omega_{\omega_0}(a)(\tilde{\varphi}_{\omega_0} + h(\omega_{\omega_0}(a),a))),\]
$\phi_{\omega_0}(a)$ is $C^2$.
Since 
\[\lambda_2(a)=\tbr{\mathbb{L}(a,\omega_0)\chi_*(a),\chi_*(a)}_{L^2},\]
we obtain
\[\begin{split}
\frac{d \lambda_2}{da} &=\omega_{\omega_0}'' -2p(p-1)\tbr{(\phi_{\omega_0})^{p-2}\frac{d\phi_{\omega_0}}{da}\frac{d\chi_*}{da},\chi_*}_{L^2}\\
&\quad -p(p-1) \Tbr{\l( (p-2)(\phi_{\omega_0})^{p-3}\l( \frac{d\phi_{\omega_0}}{da}\r)^2 + (\phi_{\omega_0})^{p-2}\frac{d^2\phi_{\omega_0}}{da^2}\r)\chi_*,\chi_*}_{L^2},
\end{split}\]
and (\ref{eigen-val}).
Finally, calculating $\frac{d^2}{da^2}\norm{\phi_{\omega_0}(a)}_{L^2}^2|_{a=0}$, we get (\ref{norm-val}).

\end{proof}

\begin{lemma}\label{main est}
Let $p\geq 2$.
Then, there exists $\omega_p>\lambda_*$ such that for $\omega_0 \in (\lambda_*,\omega_p)$, $\omega_{\omega_0}''(0)>0$ and
\[R_{p,\omega_0}
\begin{cases}
>0, \quad 2\leq p < \frac{9+\sqrt{57}}{4},\\
<0, \quad \frac{9+\sqrt{57}}{4}\leq p.
\end{cases}
\]
\end{lemma}
\begin{remark}
The first term of $R_{p,\omega_0}$ with respect to $\omega_0-\lambda_*$ yields the critical exponent $p_*$.
In Lemma \ref{main est}, we show the following expansion:
\[R_{p,\omega_0}=\frac{(-4p^2+18p-6)\pi}{3(p-1)^{3/2}(\omega-\lambda_*)^{1/2}}+O((\omega-\lambda_*)^{1/2}).\]
\end{remark}
\begin{proof}
First, we prove the positivity of $\omega''_{\omega_0}(0)$.
Let
\[\begin{split}
I_1 =&\Tbr{(\tilde{\varphi}_{\omega_0})^{p-2}\bigl(\psi_{\omega_0}\cos \frac{y}{L}\bigl)^2,\mathbb{L}_{\omega_0}^{-1}\l((\tilde{\varphi}_{\omega_0})^{p-2}\bigl(\psi_{\omega_0}\cos \frac{y}{L}\bigl)^2\r)}_{L^2}\\
 I_2=&\Tbr{\bigl(\psi_{\omega_0}\cos \frac{y}{L}\bigl)^2,(\tilde{\varphi}_{\omega_0})^{p-3}\bigl(\psi_{\omega_0}\cos \frac{y}{L}\bigl)^2}_{L^2}.
\end{split}\]
Since $(\varphi_{\omega,0}(x))^{p-3} $ is differentiable with respect to $x\in \R$ and 
\[ \Bigl| \frac{1}{\omega-\lambda_*} \l((\varphi_{\omega,0}(x))^{p-3}-\norm{\psi_*}_{L^{p+1}(\R)}^{-\frac{(p-3)(p+1)}{p-1}}(\psi_*(x))^{p-3}\r)\Bigr| \leq C (\varphi_{\theta(\omega),0}(x))^{p-4}|\partial_{\omega}\varphi_{\theta(\omega),0} (x)|,
\]
by the boundedness of $\norm{\partial_{\omega} \varphi_{\omega,0}}_{H^2(\R)}$ with respect to $\omega$ and certain upper and lower exponential decay rates for $\varphi_{\omega}$ and $\psi_{\omega}$ we have
\[
\begin{split}
I_2=&\frac{3}{8}(\omega_0-\lambda_*)^{\frac{p-3}{p-1}}\int_{\RTL}\tilde{\varphi}_{\omega_0,0}^{p-3}\psi_{\omega_0}^4 dx dy\\
=& \frac{3\pi L}{4} \norm{\psi_*}_{L^{p+1}(\R)}^{\frac{2(p+1)}{p-1}}(\omega_0-\lambda_*)^{\frac{p-3}{p-1}}\\
&+\frac{3(5p-3)\pi L}{4} \norm{\psi_*}_{L^{p+1}(\R)}^{-\frac{(p-3)(p+1)}{p-1}}(\omega_0-\lambda_*)^{\frac{p-3}{p-1}+1} \int_{\R} \psi_*^p (P_{\bot}^1H_{\lambda_*}P_{\bot}^1)^{-1}\psi_{*,p}dx\\
&+o((\omega_0-\lambda_*)^{\frac{p-3}{p-1}+1}L),            
\end{split}\]
where $\varphi_{\omega,0}=(\omega-\lambda_*)^{-1/(p-1)}\varphi_{\omega}$ and $\lambda_*<\theta(\omega) < \omega$.
On the other hand, 
\[\begin{split}
I_1&=\frac{1}{4}\tbr{(\tilde{\varphi}_{\omega_0})^{p-2}\psi_{\omega_0}^2,(L_{\omega_0}^+)^{-1}(\tilde{\varphi}_{\omega_0})^{p-2}\psi_{\omega_0}^2}_{L^2}\\
&\quad +\frac{1}{8}\tbr{(\tilde{\varphi}_{\omega_0})^{p-2}\psi_{\omega_0}^2, (L_{\omega_0}^+ + \frac{4}{L^2})^{-1}((\tilde{\varphi}_{\omega_0})^{p-2}\psi_{\omega_0}^2)}_{L^2}\\
&=I'_1+I''_1.
\end{split}\]
By $\norm{(L_{\omega_0}^+)^{-1}|_{(\psi_{\omega_0})^{\bot}}}\leq C$ and the similar calculation for $I_2$, we obtain 
\[\begin{split}
I'_1=&\frac{(\omega_0-\lambda_*)^{\frac{2(p-2)}{p-1}}}{4}\biggl \{ \Tbr{(\varphi_{\omega_0,0})^{p-2}(\psi_{\omega_0})^2,(L_{\omega_0}^+)^{-1}\l(\int_{\R}(\varphi_{\omega_0,0})^{p-2}(\psi_{\omega_0})^3dx\r)\psi_{\omega_0}}_{L^2}\\
&+ 
\Tbr{(\varphi_{\omega_0,0})^{p-2}(\psi_{\omega_0})^2,(L_{\omega_0}^+)^{-1}\l((\varphi_{\omega_0,0})^{p-2}(\psi_{\omega_0})^2-\int_{\R}(\varphi_{\omega_0,0})^{p-2}(\psi_{\omega_0})^3dx\psi_{\omega_0}\r)}_{L^2}\biggr\} \\
=& -\frac{\pi L(\omega_0-\lambda_*)^{\frac{p-3}{p-1}}\norm{\psi_*}_{L^{p+1}(\R)}^{\frac{2(p+1)}{p-1}}}{2(p-1)}\\
&+\frac{(-5p^2+9p-3)\pi L(\omega_0-\lambda_*)^{\frac{p-3}{p-1}+1}\norm{\psi_*}_{L^{p+1}(\R)}^{-\frac{(p-3)(p+1)}{p-1}}}{2(p-1)^2}\int_{\R}\psi_*^p(P_{\bot}^1H_{\lambda_*}P_{\bot}^1)^{-1} \psi_{*,p}dx\\
&+o((\omega_0-\lambda_*)^{\frac{p-3}{p-1}+1}L),\\
\end{split}\]
where $(\psi_{\omega})^{\bot}=\{u \in L^2(\R)| \tbr{u,\psi_{\omega_0}}_{L^2(\R)}=0\}$.
By the same calculation of $I_1'$ and the boundedness of $\norm{(L_{\omega_0}^++\frac{4}{L^2})^{-1}|_{(\psi_{\omega_0})^{\bot}}}$, 
\[\begin{split}
I''_1 =&\frac{(\omega_0-\lambda_*)^{\frac{2(p-2)}{p-1}}}{8}\Tbr{(\varphi_{\omega_0,0})^{p-2}(\psi_{\omega_0})^2,(L_{\omega_0}^++4/L^2)^{-1}\l(\int_{\R}(\varphi_{\omega_0,0})^{p-2}(\psi_{\omega_0})^3dx\r)\psi_{\omega_0}}_{L^2}\\
&+ \frac{(\omega_0-\lambda_*)^{\frac{2(p-2)}{p-1}}}{8}\biggl \langle (\varphi_{\omega_0,0})^{p-2}(\psi_{\omega_0})^2,\\
& \quad (L_{\omega_0}^++4/L^2)^{-1}\l((\varphi_{\omega_0,0})^{p-2}(\psi_{\omega_0})^2-\int_{\R}(\varphi_{\omega_0,0})^{p-2}(\psi_{\omega_0})^3dx\psi_{\omega_0}\r)\biggr\rangle_{L^2}\\
=&\frac{\pi L(\omega_0-\lambda_*)^{\frac{p-3}{p-1}}\norm{\psi_*}_{L^{p+1}(\R)}^{\frac{2(p+1)}{p-1}}}{12(p-1)}\\
&+\frac{(9p^2-17p+7)\pi L(\omega_0-\lambda_*)^{\frac{p-3}{p-1}+1}\norm{\psi_*}_{L^{p+1}(\R)}^{-\frac{(p-3)(p+1)}{p-1}}}{12(p-1)^2}\int_{\R}\psi_*^p(P_{\bot}^1H_{\lambda_*}P_{\bot}^1)^{-1} \psi_{*,p}dx\\
&+o((\omega_0-\lambda_*)^{\frac{p-3}{p-1}+1}L).\\
\end{split}\]
Since 
\[ \begin{split}
\frac{1}{\frac{d}{d\omega}\lambda_{\omega_0}\norm{\psi_{\omega_0}\cos \frac{y}{L}}_{L^2}^2} = \frac{1}{(p-1)\pi L} - \frac{C_*(\omega_0-\lambda_*)}{(p-1)\pi L}+o((\omega_0-\lambda_*)L^{-1}),
\end{split}\]
we obtain
\[\begin{split}
&\omega_{\omega_0}''(0)\\
=& \frac{p(p+3)(\omega_0-\lambda_*)^{\frac{p-3}{p-1}}\norm{\psi_*}_{L^{p+1}(\R)}^{\frac{2(p+1)}{p-1}}}{6}\\
&- \frac{p(2p^3+3p^2+34p-18)(\omega_0-\lambda_*)^{\frac{p-3}{p-1}+1}\norm{\psi_*}_{L^{p+1}(\R)}^{-\frac{(p-3)(p+1)}{p-1}}}{12(p-1)}\int_{\R}\psi_*^p(P_{\bot}^1H_{\lambda_*}P_{\bot}^1)^{-1} \psi_{*,p}dx\\
&+o((\omega_0-\lambda_*)^{\frac{p-3}{p-1}+1})\\
\end{split}\]
where
\[C_*=\frac{2p(2p-1)\norm{\psi_*}_{L^{p+1}(\R)}^{-(p+1)}\int_{\R}\psi_*^p(P_{\bot}^1H_{\lambda_*}P_{\bot}^1)^{-1}\psi_{*,p}dx}{p-1}.\]
Therefore, if $0<\omega_0-\lambda_*\ll 1$, then $\omega_{\omega_0}''(0)>0$.  

Next, we calculate $R_{p,\omega_0}$.
Since 
\[\begin{split}
-2\frac{d\lambda_{\omega}}{d\omega}|_{\omega=\omega_0}\norm{\psi_{\omega_0}\cos \frac{y}{L}}_{L^2}^2=-2(p-1)\pi L - 2C_*(p-1) \pi L (\omega_0 -\lambda_*) + o((\omega_0-\lambda_*)L)
\end{split}\]
and
\[ \frac{d}{d\omega}\norm{\tilde{\varphi}_{\omega}}_{L^2}^2|_{\omega=\omega_0} =\frac{4\pi L}{p-1} \norm{\psi_*}_{L^{p+1}(\R)}^{-\frac{2(p+1)}{p-1}}(\omega_0-\lambda_*)^{-\frac{p-3}{p-1}}+o((\omega_0-\lambda_*)^{\frac{2}{p-1}}L),\]
we have
\begin{equation}\label{4-e-1}
\begin{split}
R_{p,\omega_0}&= -2(p-1)\pi L - 2C_*(p-1) \pi L (\omega_0 -\lambda_*) + \frac{2p(p+3)\pi L}{3(p-1)} + o((\omega_0-\lambda_*)L)\\
&- \frac{p(2p^3+3p^2+34p-18)\pi L }{3(p-1)^2}(\omega_0-\lambda_*)\norm{\psi_*}_{L^{p+1}(\R)}^{-(p+1)}\int_{\R}\psi_*^p(P_{\bot}^1H_{\lambda_*}P_{\bot}^1)^{-1} \psi_{*,p}dx\\
=&\frac{(-4p^2+18p-6)\pi L}{3(p-1)} + o((\omega_0-\lambda_*)L)\\
&+ \frac{p(-26p^3+57p^2-82p+30)\pi L }{3(p-1)^2}(\omega_0-\lambda_*)\norm{\psi_*}_{L^{p+1}(\R)}^{-(p+1)}\int_{\R}\psi_*^p(P_{\bot}^1H_{\lambda_*}P_{\bot}^1)^{-1} \psi_{*,p}dx.
\end{split}
\end{equation}
Let 
\[ p_* =\frac{9+\sqrt{57}}{4}.\]
Since $p_*$ is the root of $-4p^2+18p-6=0$ with $p>1$, the conclusion for $p \neq p_*$ follows (\ref{4-e-1}).
Finally, we consider the case $p=p_*$.
By $p_*>4$, we have 
\[-26p_*^3+57p_*^2-82p_*+30<0.\]
Therefore, 
\begin{equation*}
\begin{split}
R_{p_*,\omega_0}
=&\frac{p_*(-26p_*^3+57p_*^2-82p_*+30)\pi L (\omega_0-\lambda_*)\norm{\psi_*}_{L^{p_*+1}(\R)}^{-(p_*+1)}}{3(p_*-1)^2}\int_{\R}\psi_*^{p_*}(P_{\bot}^1H_{\lambda_*}P_{\bot}^1)^{-1} \psi_{*,p_*}dx\\
&+ o((\omega_0-\lambda_*)L)\\
=&\frac{p_*(-26p_*^3+57p_*^2-82p_*+30)\pi L(\omega_0-\lambda_*) }{3(p_*-1)^2}\int_{\R} \psi_{*,p_*}(P_{\bot}^1H_{\lambda_*}P_{\bot}^1)^{-1} \psi_{*,p_*}dx\\
&+ o((\omega_0-\lambda_*)L)\\
\end{split}
\end{equation*}
The conclusion for $p = p_*$ follows this.
\end{proof}

Using Lemma \ref{main est} and applying the argument in Section 3 of \cite{YY2}, we obtain Theorem \ref{main result 2}.

For the completeness of the proof of Theorem \ref{main result 2}, we introduce the argument for the stability of standing with the degenerate linearized operator in \cite{MM,YY2}.
Using the following proposition, we show Theorem \ref{main result 2}.
\begin{proposition}\label{c-stability}
Let $\lambda_*<\omega_0<\omega_{*,0}$.
\begin{enumerate}
  \setlength{\parskip}{0.1cm} 
  \setlength{\itemsep}{0cm} 
\renewcommand{\labelenumi}{\rm (\roman{enumi})}
\item If $R_{p,\omega_0}>0$, then $e^{i\omega_0 t}\tilde{\varphi}_{\omega_0}$ is a stable standing wave of {\rm (NLS)} on $\RTL$ with $L=(\lambda_{\omega_0})^{-\frac{1}{2}}$.
\item If $R_{p,\omega_0}<0$, then $e^{i\omega_0 t}\tilde{\varphi}_{\omega_0}$ is an unstable standing wave of {\rm (NLS)} on $\RTL$ with $L=(\lambda_{\omega_0})^{-\frac{1}{2}}$.
\end{enumerate}
\end{proposition}

To modulate the translation symmetry for $y \in \T_L$, we define the polar coordinate $\vec{a}=(a_1,a_2)=(a \cos \frac{\tilde{a}}{L},-a\sin \frac{\tilde{a}}{L})$ for $\vec{a} \in \R^2$ and
\[\phi_{\omega_0}(\vec{a})(x,y)=\phi_{\omega_0}(a)(x,y+\tilde{a}),\quad \omega_{\omega_0}(\vec{a})=\omega_{\omega_0}(a).\]
In the following lemma, we construct a curve which captures the degeneracy of the linearized operator $S_{\omega_0}''(\tilde{\varphi}_{\omega_0})$.
\begin{lemma}\label{lem-curve}
There exist a neighborhood $U$ of $(0,0)$ in $\R^2$ and a $C^1$ function $\rho:U\to\R$ such that $\rho(0,0)=0$ and for $\vec{a} \in U$
\[Q(\phi_{\omega_0}(\vec{a})+\rho(\vec{a})\partial_{\omega}\tilde{\varphi}_{\omega_0})=Q(\tilde{\varphi}_{\omega_0}),\]
\begin{equation}\label{rho-1}
\rho(\vec{a})\tbr{\tilde{\varphi}_{\omega_0},\partial_{\omega}\tilde{\varphi}_{\omega_0}}_{L^2}=Q(\tilde{\varphi}_{\omega_0})-Q(\phi_{\omega_0}(\vec{a}))+o(\rho(\vec{a})).
\end{equation}
\end{lemma}
\begin{proof}
Since
\[\partial_{\rho} Q(\phi_{\omega_0}(\vec{a})+\rho \partial_{\omega}\tilde{\varphi}_{\omega_0})|_{\rho=0,\vec{a}=0}=\tbr{\tilde{\varphi}_{\omega_0},\partial_{\omega}\tilde{\varphi}_{\omega_0}}_{L^2} > 0,\]
the conclusion follows the implicit function theorem.
\end{proof}
Let
\[ \Phi(\vec{a})=\phi_{\omega_0}(\vec{a})+\rho(\vec{a})\partial_{\omega}\tilde{\varphi}_{\omega_0}.\]
for $ \vec{a} \in U$.

In the following lemma, we capture the degeneracy of the action $S_{\omega}$.
\begin{lemma}\label{lem-action}
For $\vec{a} \in U$,
\[S_{\omega_0}(\Phi(\vec{a}))-S_{\omega_0}(\tilde{\varphi}_{\omega_0})=\frac{\frac{d}{d \omega}\lambda_{\omega_0} \norm{\psi_{\omega_0} \cos (y/L)}_{L^2}^2 R_{p,\omega_0}}{16 \tbr{\tilde{\varphi}_{\omega_0},\partial_{\omega}\tilde{\varphi}_{\omega_0}}_{L^2}}|\vec{a}|^4+o(|\vec{a}|^4).
\]
\end{lemma}
\begin{proof}
For $\vec{a}\in U$, 
\[ \begin{split}
S_{\omega_0}(\Phi(\vec{a}))-S_{\omega_0}(\tilde{\varphi}_{\omega_0})
=&S_{\omega_{\omega_0}(\vec{a})}(\Phi(\vec{a}))-S_{\omega_0}(\tilde{\varphi}_{\omega_0})+(\omega_0-\omega_{\omega_0}(\vec{a}))Q(\tilde{\varphi}_{\omega_0})\\
=&S_{\omega_{\omega_0}(\vec{a})}(\phi_{\omega_0}(\vec{a}))-S_{\omega_0}(\tilde{\varphi}_{\omega_0})+(\omega_0-\omega_{\omega_0}(\vec{a}))Q(\tilde{\varphi}_{\omega_0})\\
&+\frac{1}{2}(\rho(\vec{a}))^2\tbr{S_{\omega_0}''(\tilde{\varphi}_{\omega_0})\partial_{\omega}\tilde{\varphi}_{\omega_0},\partial_{\omega}\tilde{\varphi}_{\omega_0}}_{L^2} + o((\rho(\vec{a}))^2).
\end{split}\]
From $\omega_{\omega_0}''(0)>0$ and (\ref{omega-ex}), $\omega_{\omega_0}(a)$ is increasing on a small interval $(0,\delta)$. 
Therefore, there exists the inverse function $a^+(\omega)$ of $\omega_{\omega_0}(a)$ form $[\omega_0, \omega_{\omega_0}(\delta))$ to $[0,\delta)$.
By the differentiability of $a^+$ for $\omega>\omega_0$, $\phi_{\omega_0}(a^+)$ is differentiable for $\omega > \omega_0$.
Thus, for $\omega,\omega_1$ with $\omega\neq \omega_1$ 
\[ \begin{split}
&\frac{S_{\omega}(\phi_{\omega_0}(a^+(\omega)))-S_{\omega_1}(\phi_{\omega_0}(a^+(\omega_1)))}{\omega - \omega_1}\\
=&\frac{\tbr{S_{\omega_1}''(\phi_{\omega_0}(a^+(\omega_1))(\phi_{\omega_0}(a^+(\omega))-\phi_{\omega_0}(a^+(\omega_1))),(\phi_{\omega_0}(a^+(\omega))-\phi_{\omega_0}(a^+(\omega_1)))}_{L^2}}{2(\omega-\omega_1)}\\
&+ Q(\phi_{\omega_0}(a^+(\omega))) + \frac{o((\phi_{\omega_0}(a^+(\omega))-\phi_{\omega_0}(a^+(\omega_1)))^2)}{\omega-\omega_1}\\
 \to & Q(\phi_{\omega_0}(a^+(\omega_1))) \quad \mbox{as } \omega \to \omega_1.\\
\end{split}\]
Moreover, since $\partial_a \phi_{\omega_0}(a)|_{a=0}= \psi_{\omega_0} \cos \frac{y}{L}$, for $\omega> \omega_0$
\[ \begin{split}
&\frac{S_{\omega}(\phi_{\omega_0}(a^+(\omega)))-S_{\omega_0}(\tilde{\varphi}_{\omega_0})}{\omega - \omega_0}\\
=&\frac{\tbr{S_{\omega_0}''(\tilde{\varphi}_{\omega_0})(\phi_{\omega_0}(a^+)-\tilde{\varphi}_{\omega_0}),(\phi_{\omega_0}(a^+)-\tilde{\varphi}_{\omega_0})}_{L^2}}{\omega_{\omega_0}''(0)(a^+)^2+o((a^+)^2)}+ Q(\tilde{\varphi}_{\omega_0}) + \frac{o((\phi_{\omega_0}(a^+)-\tilde{\varphi}_{\omega_0})^2)}{\omega_{\omega_0}''(0)(a^+)^2+o((a^+)^2)}\\
 \to & Q(\tilde{\varphi}_{\omega_0}) \quad \mbox{as } \omega \downarrow  \omega_0.\\
\end{split}\]
Hence, $S_{\omega}(\phi_{\omega_0}(a^+))$ is $C^1$ and 
\[ \frac{d S_{\omega}(\phi_{\omega_0}(a^+))}{d\omega} = Q(\phi_{\omega_0}(a^+)).\]
By the equation (\ref{norm-val}), $Q(\phi_{\omega_0}(a^+))$ is $C^1$ on $(\omega_0,\omega_{\omega_0}(\delta))$ and
\[\lim_{\omega \downarrow \omega_0}\frac{Q(\phi_{\omega_0}(a^+))-Q(\tilde{\varphi}_{\omega_0})}{\omega-\omega_0}= \frac{R_{p,\omega_0}}{2\omega_{\omega_0}''(0)}.\]
Therefore, $S_{\omega}(\phi_{\omega_0}(a^+))$ is $C^2$ with respect to $\omega$ on $(\omega_0,\omega_{\omega_0}(\delta))$ and
\begin{equation}\label{S-ex}
\begin{split}
S_{\omega}(\phi_{\omega_0}(a^+))-S_{\omega_0}(\tilde{\varphi}_{\omega_0})+(\omega_0-\omega)Q(\tilde{\varphi}_{\omega_0}) =& \frac{R_{p,\omega_0}}{4\omega_{\omega_0}''(0)}(\omega-\omega_0)^2 + o((\omega-\omega_0)^2)\\
=&\frac{\omega_{\omega_0}''(0)R_{p,\omega_0}}{16}(a^+)^4+o((a^+)^4).
\end{split}
\end{equation}
From the equation (\ref{rho-1}), we have the expansion
\begin{equation}\label{rho-ex}
(\rho(\vec{a}))^2\tbr{\tilde{\varphi}_{\omega_*},\partial_{\omega}\tilde{\varphi}_{\omega_0}}_{L^2}=\frac{(R_{p,\omega_0})^2}{16\tbr{\tilde{\varphi}_{\omega_*},\partial_{\omega}\tilde{\varphi}_{\omega_0}}_{L^2}}|\vec{a}|^4+o(|\vec{a}|^4).
\end{equation}
Since
\[S_{\omega_{\omega_0}(|\vec{a}|)}(\phi_{\omega_0}(|\vec{a}|))+(\omega_0-\omega_{\omega_0}(|\vec{a}|))Q(\tilde{\varphi}_{\omega_0}) =S_{\omega_{\omega_0}(\vec{a})}(\phi_{\omega_0}(\vec{a}))+(\omega_0-\omega_{\omega_0}(\vec{a}))Q(\tilde{\varphi}_{\omega_0}),\]
by (\ref{S-ex}) and (\ref{rho-ex}) we obtain the conclusion.
\end{proof}

We introduce the distance and tubular neighborhoods of $\tilde{\varphi}_{\omega_0}$ as follows.
Set for $\varepsilon>0$
\[ \mbox{dist}_{\omega_0}(u) = \inf_{\theta \in \R} \norm{u-e^{i \theta}\tilde{\varphi}_{\omega_0}}_{H^1},\]
\[N_{\varepsilon }=\{u \in H^1(\RTL)|\mbox{dist}_{\omega_0}(u)<\varepsilon \},\]
\[N_{\varepsilon }^0=\{ u \in N_{\varepsilon }| Q(u)=Q(\tilde{\varphi}_{\omega_0})\}.\]

Modulating the symmetry, we eliminate the degeneracy of the linearized operator around $\tilde{\varphi}_{\omega_0}$.
\begin{lemma}\label{modulation}
Let $\varepsilon >0$ sufficiently small.
Then, there exist $C^2$ function $\theta:N_\varepsilon  \to \R$, $\alpha : N_\varepsilon  \to \R$, $\vec{a}:N_\varepsilon \to U$ and $w:N_\varepsilon  \to H^1(\RTL)$ such that for $u \in N_\varepsilon $ 
\[ e^{i \theta(u)}u=\Phi(\vec{a}(u)) + w(u) + \alpha(u) \phi_{\omega_0}(\vec{a}(u)),\]
where $\tbr{w(u)+\alpha(u) \phi_{\omega_0}(\vec{a}(u)),\psi_{\omega_0} \cos (y/L)}_{L^2}=\tbr{w(u)+\alpha(u) \phi_{\omega_0}(\vec{a}(u)),\psi_{\omega_0} \sin (y/L)}_{L^2}=\tbr{w(u),\phi_{\omega_0}(\vec{a}(u))}_{L^2}=\tbr{w(u),i\phi_{\omega_0}(\vec{a}(u))}_{L^2}=0$.
\end{lemma}
\begin{proof}
Let $\psi_{\omega_0,1}=\psi_{\omega_0} \cos (y/L)$ and $\psi_{\omega_0,2}=\psi_{\omega_0} \sin (y/L)$.
We define
\[ G(u,\theta,a_1,a_2)=
\begin{pmatrix}
\tbr{e^{i\theta}u-\Phi(\vec{a}), i \phi_{\omega_0}(\vec{a})}_{L^2}\\
\tbr{e^{i\theta}u-\Phi(\vec{a}), \psi_{\omega_0,1}}_{L^2}\\
\tbr{e^{i\theta}u-\Phi(\vec{a}),  \psi_{\omega_0,2}}_{L^2}\\
\end{pmatrix},\]
where $\vec{a}=(a_1,a_2)$.
Since $G(\tilde{\varphi}_{\omega_0},0,0,0)=0$ and
\[ \frac{\partial G}{\partial (\theta,a_1,a_2)}\bigg|_{(u,\theta,a_1,a_2)=(\tilde{\varphi}_{\omega_0},0,0,0)}=
\begin{pmatrix}
\norm{\tilde{\varphi}_{\omega_0}}_{L^2}^2 & 0 &0\\
0 & -\norm{\psi_{\omega_0,1}}_{L^2}^2 &0\\
0 & 0 & -\norm{\psi_{\omega_0,2}}_{L^2}^2\\
\end{pmatrix},\]
by the implicit theorem for sufficiently small $\varepsilon >0$ there exist $C^2$ functions $\theta:N_\varepsilon \to \R$ and $\vec{a}:N_\varepsilon \to U$ such that for $u \in N_\varepsilon $
\[G(u,\theta(u),\vec{a}(u))=0.\]
We define
\[ \alpha(u) = \frac{\tbr{e^{i\theta(u)}u-\Phi(\vec{a}(u)),\phi_{\omega_0}(\vec{a}(u))}_{L^2}}{\norm{\phi_{\omega_0}(\vec{a}(u))}_{L^2}^2},\]
and
\[w(u)=e^{i\theta(u)}u-\Phi(\vec{a}(u))-\alpha(u) \phi_{\omega_0}(\vec{a}(u)).\]
Then, the conclusion follows the definition of $w$.
\end{proof}

In the following lemma, we show the estimate of $\alpha(u)$ for $u \in N_\varepsilon ^0$.
\begin{lemma}\label{alpha-est}
Let $\varepsilon >0$ sufficiently small.
There exists $C>0$ such that for $u \in N_\varepsilon ^0$,
\[|\alpha(u)|\leq C\norm{w(u)}_{L^2}(\rho(\vec{a}(u))+\norm{w(u)}_{L^2}).\]
\end{lemma}
\begin{proof}
By Lemma \ref{modulation}, for $u \in N_\varepsilon ^0$,
\[\begin{split}
Q(\tilde{\varphi}_{\omega_0})=&Q(\Phi(\vec{a}(u))+w(u)+\alpha(u) \phi_{\omega_0}(\vec{a}(u)))\\
=& Q(\tilde{\varphi}_{\omega_0}) + \alpha(u)\norm{\phi_{\omega_0}(\vec{a}(u))}_{L^2}^2+\rho(\vec{a}(u))\alpha(u)\tbr{\partial_{\omega}\tilde{\varphi}_{\omega_0},\phi_{\omega_0}(\vec{a}(u))}_{L^2} \\
&+ \rho(\vec{a}(u))\tbr{\partial_{\omega}\tilde{\varphi}_{\omega_0},w(u)}_{L^2}+Q(w(u))+(\alpha(u))^2Q(\phi_{\omega_0}(\vec{a}(u))).
\end{split}\]
Since $\rho(\vec{a}(u)) \to 0$ as $\varepsilon \to 0$, we obtain the conclusion.
\end{proof}

Next, we prove the coerciveness of the linearized operator around $\tilde{\varphi}_{\omega_0}$.
\begin{lemma}\label{coerciveness}
There exist $k_0>0$ and $\varepsilon _0>0$ such that for $a_1,a_2,\alpha \in (-\varepsilon _0,\varepsilon _0)$, if $w \in H^1(\RTL)$ with $\tbr{w,\phi_{\omega_0}(\vec{a})}_{L^2}=\tbr{w,i\phi_{\omega_0}(\vec{a})}_{L^2}=\tbr{w+\alpha \phi_{\omega_0}(\vec{a}),\psi_{\omega_0}\cos (y/L)}_{L^2}=\tbr{w+\alpha \phi_{\omega_0}(\vec{a}),\psi_{\omega_0}\sin (y/L)}_{L^2}=0$, then 
\[ \tbr{S''_{\omega_0}(\Phi(\vec{a}))w,w}_{H^{-1},H^1} \geq k_0 \norm{w}_{H^1}^2,\]
where $\vec{a}=(a_1,a_2)$.
\end{lemma}
\begin{proof}
Let $\psi_{\omega_0,1}=\psi_{\omega_0} \cos (y/L)$ and $\psi_{\omega_0,2}=\psi_{\omega_0} \sin (y/L)$.
For $w \in H^1(\RTL)$ with $\tbr{w,\phi_{\omega_0}(\vec{a})}_{L^2}=\tbr{w,i\phi_{\omega_0}(\vec{a})}_{L^2}=\tbr{w+\alpha \phi_{\omega_0}(\vec{a}),\psi_{\omega_0,1}}_{L^2}=\tbr{w+\alpha \phi_{\omega_0}(\vec{a}),\psi_{\omega_0,2}}_{L^2}=0$,
we decompose $w=b_1 \tilde{\varphi}_{\omega_0}+b_2i \tilde{\varphi}_{\omega_0}+b_3\psi_{\omega_0,1}+b_4\psi_{\omega_0,2}+w_{\bot}$, where
$\tbr{w_{\bot},\tilde{\varphi}_{\omega_0}}_{L^2}=\tbr{w_{\bot},i\tilde{\varphi}_{\omega_0}}_{L^2}=\tbr{w_{\bot},\psi_{\omega_0,1}}_{L^2}=\tbr{w_{\bot},\psi_{\omega_0,2}}_{L^2}=0$, $b_j \in \R$ for $j \in \{1,2,3,4\}$.
By the non-negativeness of $L_{\omega_0}^-$ and $L_{\omega_0}^++\lambda_{\omega_0}$, Proposition \ref{eigenvalue} and Lemma \ref{L-est}, there exists $c>0$ such that $\tbr{S_{\omega_0}''(\tilde{\varphi}_{\omega_0})w_{\bot},w_{\bot}}_{H^{-1},H^1} \geq c \norm{w_{\bot}}_{H^1}^2$, where $c$ is independent of $w_{\bot}$.
Then, from the orthogonal conditions for $w$, we have for $j\in \{1,2,3,4\}$, $b_j=O((|\vec{a}|+|\alpha|)\norm{w_{\bot}}_{L^2})$ as $|\vec{a}|+|\alpha| \to 0$.
Therefore, there exist $\varepsilon _0,k_0>0$ such that for $a_1,a_2,\alpha \in (-\varepsilon _0,\varepsilon _0)$,
\[ \begin{split}
\tbr{S''_{\omega_0}(\Phi(\vec{a}))w,w}_{H^{-1},H^1} =& \tbr{S''_{\omega_0}(\tilde{\varphi}_{\omega_0})w_{\bot},w_{\bot}}_{H^{-1},H^1}+\sum_{j=1}^4 b_j^2 + o(\norm{w_{\bot}}_{L^2}^2)\\
\geq & k_0 \norm{w}_{H^1}^2.\\
\end{split}\]
\end{proof}

In the following lemma, we investigate the variational structure of $S_{\omega_0}$ around $\tilde{\varphi}_{\omega_0}$.
\begin{lemma}\label{v-st}
Let $\varepsilon >0$ sufficiently small.
For $u \in N_\varepsilon ^0$ 
\[\begin{split}
S_{\omega_0}(u)-S_{\omega_0}(\tilde{\varphi}_{\omega_0})=&\frac{1}{2}\tbr{S_{\omega_0}''(\tilde{\varphi}_{\omega_0})w(u),w(u)}_{H^{-1},H^1}+ C_{**}R_{p,\omega_0}|\vec{a}(u)|^4\\
& +o(\norm{w(u)}_{H^1}^2)+o(|\vec{a}(u)|^4),
\end{split}\]
where $w(u)$ and $\vec{a}(u)$ are defined by Lemma \ref{modulation} and 
\[C_{**}=\frac{\frac{d}{d \omega}\lambda_{\omega_0}\norm{\psi_{\omega_0}\cos (y/L)}_{L^2}^2}{16\tbr{\tilde{\varphi}_{\omega_0},\partial_{\omega}\tilde{\varphi}_{\omega_0}}_{L^2}}.\]
\end{lemma}
\begin{proof}
Let $u \in N_\varepsilon ^0$. 
By Lemma \ref{modulation} and Lemma \ref{alpha-est}, we have
\[ \begin{split}
&S_{\omega_0}(u)-S_{\omega_0}(\tilde{\varphi}_{\omega_0})\\
=& S_{\omega_0}(\Phi(\vec{a}(u))+w(u)+\alpha(u)\phi_{\omega_0}(\vec{a}(u)))-S_{\omega_0}(\tilde{\varphi}_{\omega_0})\\
=& S_{\omega_0}(\Phi(\vec{a}(u)))-S_{\omega_0}(\tilde{\varphi}_{\omega_0})+\tbr{S_{\omega_0}'(\Phi(\vec{a}(u))),w(u)+\alpha(u)\phi_{\omega_0}(\vec{a}(u))}_{H^{-1},H^1}\\
&+\frac{1}{2}\tbr{S_{\omega_0}''(\Phi(\vec{a}(u))w(u),w(u)}_{H^{-1},H^1}+o(\norm{w(u)}_{H^1}^2).
\end{split}\]
Since $\rho(\vec{a}(u))=O(|\vec{a}(u)|^2)$ as $\mbox{dist}_{\omega_0}(u) \to 0$, $\tbr{\phi_{\omega_0}(\vec{a}(u)),w(u)}_{L^2}=0$ and $S_{\omega_0}''(\tilde{\varphi}_{\omega_0})\partial_{\omega}\tilde{\varphi}_{\omega_0}=-\tilde{\varphi}_{\omega_0}$, we have
\[ \begin{split}
&\tbr{S_{\omega_0}'(\Phi(\vec{a}(u))),w(u)}_{H^{-1},H^1}\\
=&\tbr{S_{\omega_{\omega_0}(\vec{a}(u))}'(\Phi(\vec{a}(u)))+(\omega_0-\omega_{\omega_0}(\vec{a}))\Phi(\vec{a}(u))),w(u)}_{H^{-1},H^1}\\
=&\tbr{S_{\omega_{\omega_0}(\vec{a}(u))}''(\phi_{\omega_0}(\vec{a}(u)))\rho(\vec{a}(u))\partial_{\omega}\tilde{\varphi}_{\omega_0},w(u)}_{H^{-1},H^1} + o(|\vec{a}|^4)+o(\norm{w(u)}_{H^1}^2)\\
=&\tbr{(S_{\omega_{\omega_0}(\vec{a}(u))}''(\phi_{\omega_0}(\vec{a}(u)))-S_{\omega_0}''(\tilde{\varphi}_{\omega_0}))\rho(\vec{a}(u))\partial_{\omega}\tilde{\varphi}_{\omega_0},w(u)}_{H^{-1},H^1}\\
&-\rho(\vec{a}(u))\tbr{\tilde{\varphi}_{\omega_0}-\phi_{\omega_0}(\vec{a}(u)),w(u)}_{L^2} + o(|\vec{a}(u)|^4)+o(\norm{w(u)}_{H^1}^2)\\
=& o(|\vec{a}(u)|^4)+o(\norm{w(u)}_{H^1}^2).
\end{split}\]
By Lemma \ref{alpha-est} and the continuity of $S_{\omega_0}'(\Phi(\vec{a}))$ and $\phi_{\omega_0}(\vec{a})$ at $\vec{a}=0$, we have 
\[ \tbr{S_{\omega_0}'(\Phi(\vec{a}(u))),\alpha(u)\phi_{\omega_0}(\vec{a}(u))}_{H^{-1},H^1}=o(|\vec{a}(u)|^4)+o(\norm{w(u)}_{H^1}^2).\]
Therefore, from Lemma \ref{lem-action}, we have the conclusion.
\end{proof}

\subsection{The proof of (i) of Proposition \ref{c-stability}}
In this subsection, we prove (i) of Proposition \ref{c-stability}.
Let $0<\varepsilon \ll 1$. 
By Lemma \ref{v-st} and $R_{p,\omega_0}>0$, for small $\varepsilon  $ we have that there exists $c>0$ such that for $u \in N_\varepsilon ^0$ 
\begin{equation}\label{s-est-1}
 S_{\omega_0}(u)-S_{\omega_0}(\tilde{\varphi}_{\omega_0})\geq c(\norm{w(u)}_{H^1}^2+|\vec{a}(u)|^4),
\end{equation}
where $w(u),\vec{a}(u)$ are defined by Lemma \ref{modulation}.
We suppose that there exist $\varepsilon _0>0$, a sequence $\{u_n\}_n$ of solutions and a sequence $\{t_n\}_n$ such that $t_n>0$ and $u_n(0) \to \tilde{\varphi}_{\omega_0}$ in $H^1$ and 
\[ \inf_{\theta \in \R} \norm{u_n(t_n)-e^{i\theta}\tilde{\varphi}_{\omega_0}}_{H^1} > \varepsilon _0.\]
Let
\[ v_n=\sqrt{\frac{Q(\tilde{\varphi}_{\omega_0})}{Q(u_n)}}u_n(t_n).\]
Since $Q(v_n)=Q(\tilde{\varphi}_{\omega_0})$ and $Q(u_n) \to Q(\tilde{\varphi}_{\omega_0})$ as $n \to \infty$, $\norm{v_n-u_n(t_n)}_{H^1} \to 0$ and $S_{\omega_0}(v_n)-S_{\omega_0}(\tilde{\varphi}_{\omega_0}) \to 0$ as $n \to \infty$.
By the equation (\ref{s-est-1}), $\vec{a}(u_n(t_n)) \to 0$, $\alpha(u_n(t_n)) \to 0$ and $w(u_n(t_n)) \to 0$ in $H^1$ as $n \to \infty$.
Therefore,
\[ \inf_{\theta \in \R} \norm{u_n(t_n)-e^{i\theta}\tilde{\varphi}_{\omega_0}}_{H^1}  \to 0 \quad \mbox{as } n \to \infty.\]
This is a contradiction.
We complete the proof of (i).

\subsection{The proof of (ii) of Proposition \ref{c-stability}}
In this subsection, we prove (ii) of Proposition \ref{c-stability}.
Let $0<\varepsilon \ll 1$.
We define the functions $A(u)$ and $P(u)$ as
\[A(u)=\tbr{e^{i\theta(u)}u,-i[a_1(u)\partial_{a_1}\Phi(\vec{a}(u)) + a_2(u)\partial_{a_2}\Phi(\vec{a}(u))]}_{L^2},\]
\[P(u)=\tbr{S_{\omega_{\omega_0}(\vec{a}(u))}'(u),iA'(u)}_{H^{-1},H^1},\]
for $u \in N_\varepsilon $, where $\theta(u)$ and $\vec{a}(u)$ are defined by Lemma \ref{modulation}.

Then
\[\begin{split}
A'(u)=&-ie^{-i\theta(u)}[a_1(u)\partial_{a_1}\Phi(\vec{a}(u)) + a_2(u)\partial_{a_2}\Phi(\vec{a}(u))] \\
&+ \tbr{ie^{i\theta(u)}u,-i[a_1(u)\partial_{a_1}\Phi(\vec{a}(u)) + a_2(u)\partial_{a_2}\Phi(\vec{a}(u))]}_{L^2}\theta'(u)\\
&+\tbr{e^{i\theta(u)}u,-i[\partial_{a_1}\Phi(\vec{a}(u))+a_1(u)\partial_{a_1}\partial_{a_1}\Phi(\vec{a}(u)) + a_2(u)\partial_{a_1}\partial_{a_2}\Phi(\vec{a}(u))]}_{L^2}a_1'(u)\\
&+\tbr{e^{i\theta(u)}u,-i[a_1(u)\partial_{a_1}\partial_{a_2}\Phi(\vec{a}(u)) +\partial_{a_2}\Phi(\vec{a}(u))+ a_2(u)\partial_{a_2}\partial_{a_2}\Phi(\vec{a}(u))]}_{L^2}a_2'(u),
\end{split}\]
\[\tbr{iA'(u),Q'(u)}_{L^2}=-\tbr{A'(u),iu}_{L^2}=\frac{dA(e^{i\theta}u)}{d\theta}\biggl|_{\theta=0}=0.\]
Therefore, for any solution $u(t)$ of (\ref{NLS})
\[\begin{split}
 \frac{dA(u(t))}{dt}=\tbr{A'(u(t)),-iE'(u(t))}_{H^{-1},H^1} =&\tbr{iA'(u(t)),E'(u(t))+\omega(\vec{a}(u(t)))Q'(u(t))}_{H^{-1},H^1}\\
 =&P(u(t)).
 \end{split}\]
Next, we investigate the function $P$.
\begin{lemma}\label{cal-P}
For $\vec{a} \in U$,
\[P(\Phi(\vec{a}))=-|\vec{a}|^2\rho(\vec{a})\frac{d\lambda_{\omega}}{d\omega}|_{\omega=\omega_0}\norm{\psi_{\omega_0}\cos \frac{y}{L}}_{L^2}^2+o(\rho(\vec{a})^2).\]
\end{lemma}
\begin{proof}
Let $\vec{a}_0 =(a_{1,0},a_{2,0}) \in U$.
Then $\norm{\Phi(\vec{a})}_{L^2}=\norm{\tilde{\varphi}_{\omega_0}}_{L^2}$, $\vec{a}(\Phi(\vec{a}_0))=\vec{a}_0$ and $\theta(\Phi(\vec{a}_0))=0$.
Therefore,
\begin{equation}\label{sesti}\begin{split}
 S_{\omega_{\omega_0}(\vec{a}_0)}'(\Phi(\vec{a}_0))=&S_{\omega_0}''(\tilde{\varphi}_{\omega_0})\rho(\vec{a}_0)\partial_{\omega}\tilde{\varphi}_{\omega_0}- \rho(\vec{a}_0)p(p-1)(\tilde{\varphi}_{\omega_0})^{p-2}\partial_{\omega}\tilde{\varphi}_{\omega_0} \psi_{\omega_0} \l(a_{1,0}\cos \frac{y}{L} + a_{2,0} \sin \frac{y}{L}\r) \\
&+o(\rho(\vec{a}_0)|\vec{a}_0|),
\end{split}\end{equation}
\begin{equation}\label{aesti}\begin{split}
&iA'(\Phi(\vec{a}_0))\\
=&a_{1,0}\partial_{a_1}\Phi(\vec{a}_0) +a_{2,0} \partial_{a_2}\Phi(\vec{a}_0)\\
=& \psi_{\omega_0}\l(a_{1,0}\cos \frac{y}{L} + a_{2,0}\sin \frac{y}{L}\r) + (a_{1,0}\partial_{a_1}\rho(\vec{a}_0) + a_{2,0}\partial_{a_2} \rho(\vec{a}_0))\partial_{\omega}\tilde{\varphi}_{\omega_0}\\
&+(S_{\omega_0}''(\tilde{\varphi}_{\omega_0}))^{-1}\Bigl[-|\vec{a}_0|^2\omega_{\omega_0}''(0) \tilde{\varphi}_{\omega_0} +p(p-1)(\tilde{\varphi}_{\omega_0})^{p-2}\psi_{\omega_0}^2\Bigl (a_{1,0} \cos \frac{y}{L}+a_{2,0}\sin \frac{y}{L}\Bigr)^2\Bigr]\\
&+o(|\vec{a}_0|^2)
\end{split}\end{equation}
Hence, we have
\begin{equation}\label{unst-1} \begin{split}
&P(\Phi(\vec{a}_0))\\
=&\tbr{S_{\omega_0}''(\tilde{\varphi}_{\omega_0})\rho(\vec{a}_0)\partial_{\omega}\tilde{\varphi}_{\omega_0},\psi_{\omega_0}\l(a_{1,0}\cos \frac{y}{L} + a_{2,0}\sin \frac{y}{L}\r) + (a_{1,0}\partial_{a_1}\rho(\vec{a}_0) + a_{2,0}\partial_{a_2} \rho(\vec{a}_0))\partial_{\omega}\tilde{\varphi}_{\omega_0}}_{L^2}\\
&+\tbr{\rho(\vec{a}_0)\partial_{\omega}\tilde{\varphi}_{\omega_0},-|\vec{a}_0|^2\omega_{\omega_0}''(0) \tilde{\varphi}_{\omega_0} +p(p-1)(\tilde{\varphi}_{\omega_0})^{p-2}\psi_{\omega_0}^2\Bigl (a_{1,0} \cos \frac{y}{L}+a_{2,0}\sin \frac{y}{L}\Bigr)^2}_{L^2}\\
&+\tbr{- \rho(\vec{a}_0)p(p-1)(\tilde{\varphi}_{\omega_0})^{p-2}\partial_{\omega}\tilde{\varphi}_{\omega_0} \psi_{\omega_0} \l(a_{1,0}\cos \frac{y}{L} + a_{2,0} \sin \frac{y}{L}\r),\psi_{\omega_0}\l(a_{1,0}\cos \frac{y}{L} + a_{2,0}\sin \frac{y}{L}\r)}_{L^2}\\
&+o(\rho(\vec{a}_0)|\vec{a}_0|^2)\\
=&-\rho(\vec{a}_0) \tbr{\tilde{\varphi}_{\omega_0},\partial_{\omega}\tilde{\varphi}_{\omega_0}}_{L^2}(a_{1,0}\partial_{a_1}\rho(\vec{a}_0) + a_{2,0}\partial_{a_2} \rho(\vec{a}_0)+|\vec{a}_0|^2\omega_{\omega_0}''(0) )+o(\rho(\vec{a}_0)|\vec{a}_0|^2).
\end{split}\end{equation}
By (\ref{rho-ex}), we have
\[\begin{split}
&\tbr{\tilde{\varphi}_{\omega_0},\partial_{\omega}\tilde{\varphi}_{\omega_0}}_{L^2}(a_{1,0}\partial_{a_1}\rho(\vec{a}_0) + a_{2,0}\partial_{a_2} \rho(\vec{a}_0))\\
 =& - \frac{|\vec{a}_0|^2 R_{p,\omega_0}}{2} + o(|\vec{a}_0|^2)\\
=&-|\vec{a}_0|^2\Bigl(-\frac{d\lambda_{\omega}}{d\omega}|_{\omega=\omega_0}\norm{\psi_{\omega_0}\cos \frac{y}{L}}_{L^2}^2 + \omega_{\omega_0}''(0) \tbr{\partial_{\omega}\tilde{\varphi}_{\omega_0},\tilde{\varphi}_{\omega_0}}_{L^2}\Bigr)+ o(|\vec{a}_0|^2).
\end{split} \]
Hence, the conclusion follows the equation (\ref{unst-1}).
\end{proof}

\begin{lemma}\label{cal-P-2}
Let $\varepsilon >0$ be sufficiently small and $u \in N_\varepsilon ^0 $ with $S_{\omega_0}(u)-S_{\omega_0}(\tilde{\varphi}_{\omega_0})<0$.
Then
\[P(u)=-|\vec{a}|^2\rho(\vec{a})\frac{d\lambda_{\omega}}{d\omega}|_{\omega=\omega_0}\norm{\psi_{\omega_0}\cos \frac{y}{L}}_{L^2}^2+o(\rho(\vec{a}(u))^2)+o(\norm{w(u)}_{H^1}^2).\]
\end{lemma}
\begin{proof}
By the Taylor expansion , we have 
\[\begin{split}
P(u)=&\tbr{S_{\omega_{\omega_0}(\vec{a}(u))}'(\Phi(\vec{a}(u))+w(u)+\alpha(u)\phi_{\omega_0}(\vec{a}(u))),iA'(\Phi(\vec{a}(u))+w(u)+\alpha(u)\phi_{\omega_0}(\vec{a}(u)))}_{H^{-1},H^1}\\
=& \langle S_{\omega_{\omega_0}(\vec{a}(u))}'(\Phi(\vec{a}(u)))+S_{\omega_{\omega_0}(\vec{a}(u))}''(\Phi(\vec{a}(u)))(w(u)+\alpha(u)\phi_{\omega_0}(\vec{a}(u))) ,\\
& iA'(\Phi(\vec{a}(u)))+iA''(\Phi(\vec{a}(u)))(w(u)+\alpha(u)\phi_{\omega_0}(\vec{a}(u)))\rangle _{H^{-1},H^1} + o(\rho(\vec{a}(u))^2+\norm{w(u)}_{H^1}^2)
\end{split}\]
By \eqref{sesti}, \eqref{aesti}, Lemma \ref{alpha-est} and Lemma \ref{cal-P},
\[\begin{split}
P(u)=&P(\Phi(\vec{a}(u)))+\tbr{S_{\omega_{\omega_0}(\vec{a}(u))}'(\Phi(\vec{a}(u))),iA''(\Phi(\vec{a}(u)))w(u)}_{L^2}\\
&+\tbr{S_{\omega_{\omega_0}(\vec{a}(u))}''(\Phi(\vec{a}(u)))w(u),iA'(\Phi(\vec{a}(u)))} _{H^{-1},H^1} + o(\rho(\vec{a}(u))^2+\norm{w(u)}_{H^1}^2)
\end{split}\]
By the proof of Lemma \ref{modulation}, we obtain that
\begin{equation*}
 \frac{\partial G}{\partial (\theta,a_1,a_2)} 
\begin{pmatrix}
\theta' \\
a_1'\\
a_2'
\end{pmatrix}
=
\begin{pmatrix}
-i e^{-i\theta(u)}\phi_{\omega_0}(\vec{a}(u))\\
-e^{-i\theta(u)} \psi_{\omega_0,1}\\
-e^{-i\theta(u)}\psi_{\omega_0,2}
\end{pmatrix}.
\end{equation*}
Thus $\theta'(\Phi(\vec{a}(u)), a_1'(\Phi(\vec{a}(u)))$ and $a_2'(\Phi(\vec{a}(u)))$ are linear combinations of $i\phi_{\omega_0}(\vec{a}(u))$, $\psi_{\omega_0,1}$ and $\psi_{\omega_0,2}$.
Since $\tbr{\theta'(\Phi(\vec{a}(u))),w(u)}_{L^2}=\tbr{a_1'(\Phi(\vec{a}(u))), w(u)}_{L^2}=\tbr{a_2'(\Phi(\vec{a}(u))),w(u)}_{L^2}=O(\alpha(u)\norm{w(u)}_{H^1})$,
we have
\[iA''(\Phi(\vec{a}(u)))w(u)=O(\alpha(u)\norm{w(u)}_{H^1}).\]
Therefore, by the orthogonal condition of $w(u)$ and $A'(\Phi(\vec{a}(u)))=O(\vec{a}(u))$ we obtain
\[\begin{split}
P(u)=&P(\Phi(\vec{a}(u)))+(\omega_{\omega_0}(\vec{a}(u))-\omega_0)\tbr{w(u),iA'(\Phi(\vec{a}(u)))}_{L^2}\\
&+\tbr{p(|\Phi(\vec{a}(u))|^{p-1}-|\tilde{\varphi}_{\omega_0}|^{p-1})w(u),iA'(\Phi(\vec{a}(u)))}_{L^2}\\
&+\tbr{S_{\omega_0}''(\tilde{\varphi}_{\omega_0})w(u),\psi_{\omega_0}\l(a_{1}(u)\cos \frac{y}{L} + a_{2}(u)\sin \frac{y}{L}\r) }_{H^{-1},H^1}\\
&+\tbr{S_{\omega_0}''(\tilde{\varphi}_{\omega_0})w(u),(a_{1}(u)\partial_{a_1}\rho(\vec{a}(u)) + a_{2}(u)\partial_{a_2} \rho(\vec{a}(u)))\partial_{\omega}\tilde{\varphi}_{\omega_0}}_{H^{-1},H^1}\\
&+\tbr{w(u),-|\vec{a}(u)|^2\omega_{\omega_0}''(0) \tilde{\varphi}_{\omega_0} +p(p-1)(\tilde{\varphi}_{\omega_0})^{p-2}\psi_{\omega_0}^2\Bigl (a_{1}(u) \cos \frac{y}{L}+a_{2}(u)\sin \frac{y}{L}\Bigr)^2}_{H^{-1},H^1}\\
&+ o(\rho(\vec{a}(u))^2+\norm{w(u)}_{H^1}^2)\\
=&P(\Phi(\vec{a}(u)))+ o(\rho(\vec{a}(u))^2+\norm{w(u)}_{H^1}^2)
\end{split}\]
Hence, we obtain the conclusion.
\end{proof}

We assume $e^{i\omega_0 t}\tilde{\varphi}_{\omega_0}$ is stable.
Let $\{\vec{a}_n\}_n$ be a sequence with $\vec{a}_n \to 0$ and $\{u_n\}_n$ be the sequence of solutions with $u_n(0)=\Phi(\vec{a}_n)$.
Since $R_{p,\omega_0}<0$ and there exists $C>0$ such that 
\[S_{\omega_0}(\Phi(\vec{a}_n))-S_{\omega_0}(\tilde{\varphi}_{\omega_0})=CR_{p,\omega_0}|\vec{a}_n|^4+o(|\vec{a}_n|^4),\]
we obtain $S_{\omega_0}(\tilde{\varphi}_{\omega_0})>S_{\omega_0}(\Phi(\vec{a}_n))$ for sufficiently large $n>1$.
From Lemma \ref{v-st} and Lemma \ref{cal-P-2} we have for sufficiently large $n>1$  
\[\begin{split}
0<&S_{\omega_0}(\tilde{\varphi}_{\omega_0})-S_{\omega_0}(\Phi(\vec{a}_n))\\
=&S_{\omega_0}(\tilde{\varphi}_{\omega_0})-S_{\omega_0}(u_n(t))\\
\leq & -C_{**}R_{p,\omega_0}|\vec{a}(u_n(t))|^4 - \frac{k_0}{2}\norm{w(u_n(t))}_{H^1}^2+o(\norm{w(u_n(t))}_{H^1}^2)+o(|\vec{a}(u_n(t))|^4).
\end{split}\]
By the stability of $e^{i\omega_0 t}\tilde{\varphi}_{\omega_0}$ and the equation (\ref{rho-ex}), we obtain there exists $c>0$ such that for sufficiently large $n>1$
\[0<S_{\omega_0}(\tilde{\varphi}_{\omega_0})-S_{\omega_0}(\Phi(\vec{a}_n))\leq cP(u_n(t)).\]
Since $\rho(\vec{a}(u_n(t)))$ is positive and bounded for $t\geq 0$ and sufficiently large $n>1$, there exists $\delta>0$ such that for $t\geq 0$
\[\frac{dA(u_n(t))}{dt}=P(u_n(t))>\delta.\]
This contradicts the boundedness of $A$ on $N_\varepsilon $.
Hence, $e^{i\omega_0 t}\tilde{\varphi}_{\omega_0}$ is unstable.

\section*{Acknowledgments}
The author would like to express his great appreciation to Professor Yoshio Tsutsumi for a lot to helpful advices and encouragements.
The author would like to thank Professor Mashahito Ohta for his helpful indication.

Yohei Yamazaki

Department of Mathematics

Kyoto University

Kyoto 606-8502

Japan

E-mail address: y-youhei@math.kyoto-u.ac.jp
\end{document}